\documentclass[a4paper,10pt]{amsart}
\usepackage[utf8]{inputenc}
\usepackage[english]{babel}

\usepackage{amsfonts}

\usepackage[english]{babel}
\usepackage{amsmath,amssymb,verbatim,mathrsfs,latexsym,paralist,graphics}

\usepackage{amsfonts,mathtools,url, enumerate} 
\usepackage{subcaption}
\usepackage{float}
\usepackage{tikz-cd}
\usepackage{hyperref}
\usepackage{indentfirst}
\usepackage{multirow}
\usepackage[a4paper]{geometry}
\usepackage{amsthm}
\allowdisplaybreaks[4]
\usepackage{soul}
\usepackage{ tipa }
\usepackage{tikz}

\usepackage{enumitem}
\usepackage{url} 
\usepackage[all]{xy}

\newcommand{\Z}{\mathbb{Z}}

\newcommand{\R}{\mathbb{R}}
\newcommand{\N}{\mathbb{N}}
\newcommand{\qdim}{\mathrm{qdim}}

\newcommand{\mc}{\mathcal}

\newcommand{\od}{\mathrm{out}}
\newcommand{\id}{\mathrm{in}}

\numberwithin{equation}{section}
\newtheorem{theorem}{Theorem}[section]
\newtheorem{proposition}[theorem]{Proposition}

\newtheorem{lemma}[theorem]{Lemma}
\newtheorem{corollary}[theorem]{Corollary}
 
\newtheorem{defn}[theorem]{Definition}    

\theoremstyle{remark}
 
\theoremstyle{remark}
\newtheorem{rmk}[theorem]{Remark}
\newtheorem{example}[theorem]{Example}

\DeclareMathOperator{\Mult}{\mathscr{M}}
\DeclareMathOperator{\SubGraph}{\mathscr{S}ub}

\newcommand{\mpslash}{\raisebox{-.2em}{\begin{tikzpicture}%
\draw[thick] (-.075,-.2) -- (.1,.13);%
\begin{scope}[transform canvas={scale=0.5}]%
\node at (0.2,-.45) {MP\ };%
\end{scope}\end{tikzpicture}}\ }

\newcommand{\mpslashpedice}{\raisebox{-.2em}{\begin{tikzpicture}[scale=.75]%
\draw[thick] (-.075,-.2) -- (.1,.1);%
\begin{scope}[transform canvas={scale=0.35}]%
\node at (0.25,-.45) { \, MP\ };%
\end{scope}\end{tikzpicture}}\  }

\title[Multipath matroids]{Multipath matroids, digraph colourings, and the Tutte polynomial}
\author{Luigi Caputi, Carlo Collari, Sabino Di Trani}
\date{}

\begin{document}
\maketitle

\begin{abstract}
We characterise the digraphs for which the multipaths, that is disjoint unions of directed paths, yield a matroid.
For such graphs, called MP-digraphs, we prove that the Tutte polynomial of the multipath matroid is related to counting certain digraph colourings. 
Finally, we prove that, for MP-forests, the decategorification of the multipath cohomology yields a specialisation of the Tutte polynomial.
\end{abstract}

\section{Introduction} 

Multipaths are disjoint unions of directed paths in a digraph. These objects were introduced in~\cite{turnerwagner} to define an analogue of chromatic homology~\cite{Helme-Guizon-Rong} for directed graphs, while preserving the fact that the homology of (coherently oriented) cyclic (di)graphs recovers (a truncation of) Hochschild homology~\cite{Prz}. With the aim of investigating combinatorial properties of directed graphs by homological means, the authors studied in~\cite{primo} what they called multipath cohomology. 
The theories in \cite{turnerwagner, primo} can be seen as homology theories for posets applied to the \emph{path poset}, i.e.~the set of multipath ordered by inclusion. 
The order complex of the path poset is intimately related with other, well-known, simplicial complexes associated to (undirected) graphs, such as matching complexes \cite{spri2, ramos} and to cycle-free chessboard complexes \cite{VZ}.

In \cite{primo, secondo, terzo}, it was highlighted how combinatorial properties of digraphs reflect on the structure of path posets, and on the topological properties of the associated order complex. Some of the properties of the path posets are somehow reminiscent of other combinatorial structures.
During the conference ``Geometry, Algebra and Combinatorics of Moduli Spaces and Configurations V'' held in 2023 in Dobbiaco, some among the authors were asked about the relationship between matroids and path posets.
Our first result in this direction is a complete characterisation of the digraphs for which the set of multipaths $\Mult$, each of which is seen as the set of its edges, satisfies the axioms of set of independents of a matroid -- whose ground set is $E$, the set of edges of the digraph.

\begin{theorem}\label{thm:intro-multipath-matroid}
Let $G$ be a (finite) digraph, with possibly a finite number of loops on each vertex. Then, the digraph $G$ satisfies both:
\begin{enumerate}[label={\rm (\arabic*)}]
\item  $G$ does not contain (up to reversing the orientation of all edges) one of the digraphs in Figure~\ref{fig:avoiding} as subgraphs,
\item  every coherently oriented cycle 
(loops excluded)
in $G$ is a connected component,
\end{enumerate} 
if and only if the pair $(E(G), \Mult(G))$ is a matroid.
\end{theorem}

Any digraph $G$ as above which satisfies  both (1) and (2) is called \emph{MP-digraph}. If $G$ is an MP-digraph, we call~$M_G = (E(G), \Mult(G))$ the \emph{multipath matroid} (associated to $G$).
As a direct consequence, if $G$ is an MP-digraph the order complex of its path poset coincides with the matroid complex of $M_G$. Therefore, this complex is pure  and shellable. This should be compared with the computations of \cite{terzo}. 

We investigate further properties of multipath matroids. We show that they are graphic, regular and that, in fact, are direct sums of uniform matroids. 
In addition, we define the MP-contraction of directed graphs. 
It turns out that the family of MP-digraphs is closed under deletion and MP-contraction.
Furthermore, deletion and MP-contraction of MP-digraphs correspond to the deletion and contraction, respectively, of the associated multipath matroids (Proposition~\ref{prop:isodelcontr}).
This is used to obtain a ``deletion/contraction''-type relation for Tutte polynomials of multipath matroids (Corollary~\ref{cor:tuttegraph}). 

 A celebrated theorem of Tutte~\cite{tutte} shows that the chromatic polynomial, and thus the number of colourings, of a graph can be obtained by specializing the Tutte polynomial of the associated graphic matroid. It is then natural to ask whether a similar result is true for MP-digraphs. The analogue of colourings in our setting are flowing colourings (Definition~\ref{defn:Dcolourings}). 
 Then, the Tutte polynomial $T_{M_G}$ of the multipath matroid $M_G$ is related to the number $\tau_G(k)$ of flowing $k$-colourings of any suitable spanning forest in $G$ (Corollary~\ref{cor:tuttecolouring}).

\begin{theorem} 
Let $G$ be an MP-digraph without coherently oriented cycles and let $M_G$ be the associated multipath matroid. Then, for every spanning forest $S$ for $G$ of maximal rank, we have 
\begin{equation*} k^{p_0(G)}T_{M_G}(1-k,0)=(-1)^{r(G)}\tau_S(k) \ ,\end{equation*}
where $p_0(G)$ is the number of connected components of $G$, and $r$ is the rank function of $M_G$.
\end{theorem}

Chromatic homology is an homology theory for (undirected) graphs introduced by Helme-Guizon and Rong~\cite{Helme-Guizon-Rong}, and served as inspiration for the definition of multipath cohomology.
The key property of chromatic homology is that its (graded) Euler characteristic recovers the chromatic polynomial of a graph.
Since chromatic homology -- as well as multipath homology -- is functorial in an appropriate sense, we say that it provides a categorification of the chromatic polynomial.  
Relating classical graph invariants and homology theories for graphs is a vibrant area of research. In recent years, a number of such relations have been discovered and studied; for instance, some classical invariants of graphs which have been categorified are: Tutte polynomial of graphs \cite{zbMATH05118586}, dichromatic polynomial for graphs \cite{zbMATH05313438}, characteristic polynomial of matroids \cite{saito2024categorificationcharacteristicpolynomialmatroids}, chromatic symmetric polynomial \cite{zbMATH06804238}, (evaluations of) the connected domination polynomial \cite{zbMATH07681103}, and magnitude~\cite{zbMATH06826922}, to name a few. 
It is an open question whether the graded Euler characteristic of multipath cohomology is related to any known digraph invariant (see \cite[Question~8.3]{primo}). One of the aims of this paper is to positively answer to this question for MP-forests. 
Given an integral domain $R$, denote by $\chi_{\mu}(G;\alpha)$ be the graded Euler characteristic of the multipath homology computed with respect to the ($\Z$-)graded $R$-algebra $A$ whose graded dimension is $\alpha\in \Z[q,q^{-1}]$ -- see Section~6, and \cite{primo} for the detailed definitions.

\begin{theorem}\label{thm:intro-chi} Let $G$ be an MP-forest. Then,
\begin{equation*} \chi_\mu (G, \alpha) = (-1)^{r(G)} \alpha^{t(G)}T_{M_G}(1-\alpha, 1)  \ , \end{equation*}
where $t(G) = |E(G)| - r(G)+ p_0(G)$.
\end{theorem}

Despite the fact that Theorem~\ref{thm:intro-chi} does not provide a complete answer to \cite[Question~8.3]{primo}, it does provide a starting point to understand the combinatorial meaning of $\chi_\mu$.

\subsection*{Conventions}
Unless otherwise specified, all (directed) graphs and matroids are \emph{finite}.
Digraphs do not have multiple edges, but multiple loops at the same vertex are allowed.

\subsection*{Acknowledgements:} The authors are grateful to INdAM-GNSAGA. LC was supported by the Starting Grant 101077154 “Definable Algebraic Topology” from the ERC. CC acknowledges the MUR-PRIN project 2022NMPLT8 and the MIUR Excellence Department Project awarded to the Department of Mathematics, University of Pisa, CUP I57G22000700001.
For the valuable suggestions and motivation, the authors are thankful to Michele D'adderio, Roberto Pagaria, Lorenzo Vecchi and Dobbiaco Tutt\textschwa.  
SDT would like to express his sincere gratitude to Roberto Pagaria for the many insightful conversations about matroids, and their applications. 
SDT also extends his thanks to Alessandra Caraceni and Michele D'adderio for patiently answering a wide range of combinatorial questions during his stays in Pisa. 

\section{Basics on Matroids}
In this section we recall the definition of (finite) matroids, some of their properties, and the definition of the Tutte polynomial.  
We refer to \cite{OxMat} and \cite{WhiteMaTh} for a complete discussion about matroids.

Let $X$ be a finite set and $\mc{I}$ a family of subsets of $X$. Given a set $A$, we denote by $|A|$ its cardinality. We say that $\mc{I}$ is \emph{a family of independent sets} for $X$ if the following are satisfied: 
\begin{enumerate}[label = {\rm (I\arabic*)}]
\item  the empty set is in $\mc{I}$;
\item  if $A\in \mc{I}$ and $B \subset A$, then $B \in \mc{I}$;
\item \label{I3}
if $A,B \in \mc{I}$ and $|A| > |B|$ then there exists $x \in A$ such that $B \cup \{x\} \in \mc{I}$.
\end{enumerate}
 The elements of $\mc{I}$ are called \emph{independent sets} or, simply, \emph{independents}.

\begin{defn}\label{defn:matroid}
A (finite) matroid $M$ is a pair $(X, \mc{I})$, where $X$ is a finite set and $\mc{I}$ is  a family of independent sets for $X$.
\end{defn}
We refer to the set $X$ as the \emph{ground set} of the matroid~$M$. A maximal independent set is called a \emph{basis} for $M$. All bases of a matroid   have the same cardinality. A subset of $X$ that is not an independent is called a \emph{dependent set}. A minimal dependent set is called a \emph{circuit}.
The set $\mc{C}(M)$ of circuits of a matroid $M$ satisfies the following important properties:
\begin{itemize}
\item for every $A,B \in \mc{C}(M)$, if  $A \subset B$, then $A=B$; 
\item if $x \in A \cap B$ and $A \neq B$, then there exists $C \in \mc{C}(M) $ such that $C \subset A \cup B \setminus \{x\}$.
\end{itemize}
Using circuits, it is possible to provide an alternative definition of matroid. In fact, we have the following;

\begin{rmk}\label{defn:matroidcircuits}  A matroid can be viewed as a pair of sets $(X, \mc{C})$, where $X$ is the ground set and $\mc{C}$ is a family of elements of $\mc{P}(X)$ such that: 
\begin{enumerate}[label={(C\arabic*)}]
\item\label{C1}  the empty set is \emph{not} in $\mc{C}$;
\item \label{C2}  if $A, B \in \mc{C}$, then $A \subset B$ implies $A=B$; 
\item\label{C3} if $A, B \in \mc{C}$, $x \in A \cap B$ and $A \neq B$, then there exists $C \in \mc{C}$ such that $C \subset A \cup B \setminus \{x\}$.
\end{enumerate}
The two definitions of matroid are equivalent; to see it, it is sufficient to define
\[ \mc{I}_{\mc{C}} = \{ S \subseteq X \mid C \nsubseteq S,\text{ for each }C\in\mc{C} \}\ ,\]
and this is easily seen to be a family of independent sets -- see also~\cite[Section~1.1]{OxMat}.
\end{rmk}

Given a matroid  $M=(X,\mc{I})$, its \emph{rank function}  $r_M\colon \mc{P}(X) \rightarrow \mathbb{N}_{\geq 0}$ is the function which associates to each $S\subseteq X$ the size of a maximal independent set contained in $S$. 

A primer example of matroids is given by sets of independent vectors.
Let $V$ be a $\mathbb{C}$-vector space and $X \subseteq V$ be a finite set. Then, $X$ has a natural structure of matroid: $I \subset X$ is independent if and only if the vectors in $I$ are linearly independent. In this context, the value of the rank function on $S\subseteq X$ is the dimension of the vector space spanned by $S$.
Let us see another example, which will be used in the follow-up.

\begin{example}[Uniform Matroids]\label{ex:uniform matroid}
Consider the set $X = \{ 1,..., n\} \subseteq \N$ and
let $P_k(X)$ be the set of subsets of $X$ with cardinality at most $k \leq |X|=n$. Then, $\mc{P}_k(X)$ satisfies the conditions of being an independent set, and the pair $(X,\mc{P}_k(X))$ is a matroid $U_{k,n}$ called the \emph{uniform matroid}. If $A $ is a subset of $ X$ then $r_{U_{k,n}}(A)=\min\{|A|, k\}$.
\end{example}

Matroids yield a category, when considering the following notion of morphisms --  usually called strict maps. 
 
A morphism between the matroids $M=(X, \mc{I})$ and $M'=(X', \mc{I}')$ is a map $f \colon X \rightarrow X'$ such that  $f^{-1}(A') \in \mc{I}$ for all $A' \in \mc{I}'$.  
Two matroids $M=(X, \mc{I})$ and $M'=(X', \mc{I}')$ are isomorphic, and in such case we write $M\cong M'$, if there exists a bijective map $f \colon X \rightarrow X'$ such that  $f^{-1}(A') \in \mc{I}$ if and only if $A' \in \mc{I}'$.
The category of matroids is closed under coproducts; 
\begin{defn}[Direct Sum Matroid]\label{defn:summat} Let $M_1=(X_1, \mc{I}_1)$ and $M_2=(X_2, \mc{I}_2)$ be  matroids. The direct sum $M_1 \oplus M_2$ is the matroid having ground set $X_1 \sqcup X_2$ and as family of independent sets $\mc{I} = \{ I_1 \sqcup I_2 \mid I_1\in \mc{I}_1 \text{ and } I_2 \in \mc{I}_2\}$.
\end{defn}

Furthermore, there are two natural operations on matroids: deletion and contraction, which we now recall.
Consider a matroid $M=(X,\mc{I})$ and $x \in X$ any element.
\begin{defn}[Deletion Matroid]\label{defn:deletionmat}  The deletion of $M=(X,\mc{I})$ with respect to $x$ is the matroid $M \setminus x=(X', \mc{I'})$ defined by the data: 
\begin{itemize}
\item[DM1:] the ground set $X'$ is the set $X \setminus \{x\}$; 
\item[DM2:] the independent set $\mc{I'}$ is the sets of $I \in \mc{I}$ such that $I \subset X'$.
\end{itemize}
\end{defn}
Assume further that $\{x\}$ is an independent for $M=(X,\mc{I})$.
\begin{defn}[Contraction Matroid]\label{defn:contractionmat}  The matroid $M / x$ obtained contracting $x$ is defined by: 
\begin{itemize}
\item[CM1:] the ground set $X$ is the set $X \setminus \{x\}$; 
\item[CM2:] the independent set $\mc{I'}$ is given by the sets $I \subset X'$ such that   $I \cup \{x\} \in \mc{I}$. 
\end{itemize}
\end{defn}
 
Let $M = (X,\mc{I})$ be a matroid with rank function $r$.
The \emph{corank} of a subset $A \subset X$ is the difference $z(A)\coloneqq r(M) -r(A)$. The difference $n(A)\coloneqq |A|-r(A)$ is called the \emph{nullity} of $A$. 
Using the corank and the nullity functions, we can now recall the definition of Tutte polynomial for a matroid. For an extended discussion about the Tutte polynomials we refer to \cite[Chapter~6]{WhiteMatApp}. 

\begin{defn}\label{def:Tutte} Let $M=(X,I)$ be a finite matroid. Its Tutte polyonomial is defined by the formula: 
\[T_M(x,y) = \sum_{A \subseteq X} (x-1)^{z(A)}(y-1)^{n(A)} \ .\]
By convention, if $X = \emptyset $ then $T_M(x,y) = 1$.
\end{defn}

The Tutte polynomial encodes many nice combinatorial properties of $M$, and it is well-behaved with respect to direct sums; let $M_1=(X_1, \mc{I}_1)$ and $M_2=(X_2, \mc{I}_2)$ be  matroids, then \begin{equation}\label{prop:DSTutte}
    T_{M_1 \oplus M_2}(x,y)=T_{M_1}(x,y) \, T_{M_2}(x, y) \ .
\end{equation}
The Tutte polynomial can be recursively computed using  deletions and contractions.
We say that  $e \in X$ is a \emph{loop} for a matroid $M=(X,I)$ if $\{e\}$ is not an independent set. Moreover, we say that  $e \in X$ is a \emph{coloop} (or an isthmus) of $M$ if $e$ is contained in every independent set.
The following result is well-known -- see, for instance~\cite[Chapter~6]{WhiteMatApp};
\begin{theorem}
\label{thm:Tutteprop}  If $e$ is a loop of $M$, then 
\[T_M(x,y)= yT_{M \setminus e} (x,y)\ .\]
If $e$ is a coloop, we have 
\[T_M(x,y)= xT_{M / e} (x,y)\ .\]
Finally, we have
\[T_M(x,y)= T_{M \setminus e} (x,y) + T_{M / e} (x,y) \ , \]
if $e$ is neither a loop nor a coloop for $M$.
\end{theorem}

\subsection{Graphic Matroids} 
In this subsection we focus on an important class of matroids, called graphic matroids.
Let $G$ be an undirected finite graph with vertex set $V(G)$ and edges $E(G)$. Let $\mathcal{C}$ be the set of subgraphs of $G$ containing at least a cycle. 
By identifying each element of $\mathcal{C}$ with its set of edges, the set $\mathcal{C}$ satisfies the axioms of circuits of a matroid given in Remark~\ref{defn:matroidcircuits}, over the ground set $X=E(G)$.

\begin{defn}[Graphic Matroid] 
Let $G$ be a finite graph, and $\mc{C}$ the set of its subgraphs containing at least one cycle. The matroid  $(E(G), \mc{I}_{\mc{C}})$ is the graphic matroid associated to $G$.
\end{defn}

In other words, the independent sets of the graphic matroid associated to $G$ are (the edges of) subforests of $G$.

\begin{example}\label{es:matrees} Let $G$ be a tree with $n+1$ vertices. Since $G$ has no cycles, its graphic matroid $(E(G), P(E(G))$ can be identified with the uniform matroid $ U_{n,n}$. In particular, all trees with the same number of vertices have isomorphic graphic matroids.
\end{example}

There are classical results about graphic matroids relating the Tutte polynomial of a graphic matroid to properties of the underlying graph, as we now shall recall. 
We say that  a (vertex) $k$-colouring for an \emph{unoriented} graph $G=(V(G),E(G))$ is a map $c \colon V(G) \rightarrow \{1, \dots , k\}$ such that if $(v_1, v_2) \in E(G)$ then $c(v_1) \neq c(v_2)$.

The chromatic function of $G$ is the function $\chi_G\colon \N \rightarrow \N$ that to each non-negative integer $k$ associates the number of $k$ colourings of $G$.

For a graph $G$, a loop is an edge adjacent to a single vertex. Then, we have the following;
\begin{theorem}\label{thm:Tuttecolourings} 
Let $G$ be a graph without loops, and let $M$ be its associated graphic matroid. Then, for any $k\in \mathbb{N}$, we have 
\[k^{p_0(G)}T_{M}(1-k,0)=(-1)^{  r(G)  }\chi_G(k)\]
where $p_0(G)$ is the number of connected components of $G$.
\end{theorem}
We refer to \cite{OxGraphic} for an extended discussion about graphic matroids, and to \cite[Chapters 4 \& 6]{WhiteMatApp} for further details on the Tutte polynomials of graphic matroids and the proof of Theorem~\ref{thm:Tuttecolourings}. 
We conclude this section with a few words about representability of matroids.
 
A finite matroid $M=(X,\mc{I})$ is representable over a field $\mathbb{F}$ if there exists a matrix $N$ with coefficients in $\mathbb{F}$ and a bijective map $f$ from $X$ to the set of columns of $N$ such that $A \in I$ if and only if $f(I)$ is a set of linearly independent sets of $N$.

\begin{defn}[Binary Matroid]\label{def:binary} A matroid is binary if it is representable over the field with two elements $\mathbb{F}_2$.
\end{defn} 
We say that a matroid is \emph{regular} if it is representable over any field.
 Every graphic matroid regular -- see, for instance,~\cite[Proposition 5.1.2]{OxMat}.

\section{Path Poset and Multipath Cohomology}
In this section,  we review some basic notions on directed graphs, and recall the definition of path posets as developed in~\cite{turnerwagner, primo, secondo}. Following~\cite{terzo}, we recall also the notion of dynamical regions of digraphs.

\subsection{Digraphs}

By a \emph{directed graph}~$  G$, often shortened to \emph{digraph},  we mean a pair of finite sets $(V(G),E(G))$ of  \emph{vertices} and \emph{edges}.
Each edge $e\in E(G)$ has a \emph{source} $s(e)\in V(G)$ and a \emph{target}~$t(e)\in V(G)$. For any pair of \emph{distinct} vertices, say $v$ and $w$, we ask that there is at most one edge such that $v = s(e)$ and $w = t(e)$ -- which will be denoted simply by $(v,w)$. An edge such that $s(e) = t(e)$ is called \emph{loop}. For each vertex $v$, there might be more than one loop having~$v$ both as target and source.
 
If a vertex $v$ is {either} a source or a target of an edge $e$, we will say that \emph{$e$ is incident to $v$}. Furthermore, we say that $v \in V(  G)$ is a \emph{sink} (resp.~a \emph{source}) if for every $e\in E(G)$ incident to $v$ we have $v = t(e)$ (resp.~$v = s(e)$) and $e$ is not a loop. 
A digraph~$  G$ with~$n$ edges is a \emph{sink} (resp. a \emph{source}) \emph{on $n+1$ vertices} if it has a unique sink (resp.~source), and every edge in $  G$ is incident to it.  
If $G$ is a digraph, we denote by $G^{\mathrm{ud}}$ the underlying undirected multigraph. The valence of a vertex $v$ in a digraph $G$ is defined as the valence of $v$ in $G^{\mathrm{ud}}$. A \emph{morphism of digraphs} from~$  G_1$ to  $  G_2$ is a function $\phi\colon V(  G_1)\to V(  G_2)$ such that:
\[ e = (v,w) \in E(  G_1)\ \Longrightarrow\ \phi (e) \coloneqq (\phi(v),\phi(w)) \in E(  G_2)\ .\]
 A morphism of digraphs is called \emph{regular} if it is injective as a function. A \emph{subgraph}    $H$ of a digraph   $G$ is a digraph with $V(H)\subseteq V(G)$ and $E(H)\subseteq E(G)$; in such case, we  write $   H \leq   G$.

\begin{rmk}\label{rem:digons}
Coherently oriented cycles of length two in a digraph $G$, i.e.~subgraphs on two vertices $v,w$, and edges $(v,w),(w,v)$,   have no common edges. 
\end{rmk}

\begin{figure}[h]
	\centering
	\begin{subfigure}[b]{0.4\textwidth}
		\centering
		\begin{tikzpicture}[baseline=(current bounding box.center)]
		\tikzstyle{point}=[circle,thick,draw=black,fill=black,inner sep=0pt,minimum width=2pt,minimum height=2pt]
		\tikzstyle{arc}=[shorten >= 8pt,shorten <= 8pt,->, thick]
		
		\node[above] (v0) at (0,1) {$v_0$};
		\draw[fill] (0,1)  circle (.05);
		\node[above] (v1) at (1.5,1) {$v_1$};
		\draw[fill] (1.5,1)  circle (.05);
		\node[above] (v2) at (3,1) {$v_{2}$};
		\draw[fill] (3,1)  circle (.05);

		\draw[thick,   -latex] (1.35,1) -- (0.15,1);
		\draw[thick,  -latex] (1.65,1) -- (2.85,1);
	\end{tikzpicture}
\caption{A linear source. }\label{fig:source}
\end{subfigure}
\hspace{0.1\textwidth}
	\begin{subfigure}[b]{0.4\textwidth}
		\centering
		\begin{tikzpicture}[baseline=(current bounding box.center)]
		\tikzstyle{point}=[circle,thick,draw=black,fill=black,inner sep=0pt,minimum width=2pt,minimum height=2pt]
		\tikzstyle{arc}=[shorten >= 8pt,shorten <= 8pt,->, thick]
		
		\node[above] (v0) at (0,1) {$v_0$};
		\draw[fill] (0,1)  circle (.05);
		\node[above] (v1) at (1.5,1) {$v_1$};
		\draw[fill] (1.5,1)  circle (.05);
		\node[above] (v2) at (3,1) {$v_{2}$};
		\draw[fill] (3,1)  circle (.05);

		\draw[thick,  -latex] (0.15,1) -- (1.35,1);
		\draw[thick,  -latex] (2.85,1) -- (1.65,1);
	\end{tikzpicture}
\caption{A linear sink. }\label{fig:sink}
\end{subfigure}
	\caption{Linear source and sink.}\label{fig:nnstep}
\end{figure}
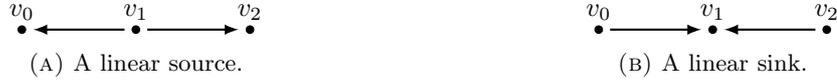

If $H\leq   G$ and $V(   H) = V(  G)$ we  say that $H$ is a \emph{spanning subgraph} of $  G$.
Given a proper spanning subgraph $   H <   G$, we can find an edge $e\in E(  G)\setminus E(   H)$. The spanning subgraph of   $G$ obtained from $H$ by adding an edge $e$  is simply denoted by  $H\cup e$. 
\subsection{Multipaths and Path posets}
In this subsection we introduce  one of the main tools in the definition of multipath cohomology of directed graphs: the \emph{path poset}~\cite{primo}.

By a \emph{simple path} in a digraph $G$ we mean a sequence of edges $e_1,\dots,e_n$ of $G$, which are not loops, such that~$s(e_{i+1})=t(e_i)$ for $i=1,\dots,n-1$, no vertex is encountered twice, i.e. if~$s(e_i) = s(e_j)$ or $t(e_i) = t(e_j)$, then $i=j$, and it does not form a cycle, i.e.~$s(e_1)\neq t(e_n)$.
A connected component of $G$ is a subgraph $H$ of   $G$ whose geometric realisation (as CW-complex) is a connected component of the geometric realisation of $G$.
We are interested in disjoint unions of simple paths; following~\cite{turnerwagner}, we call them multipaths:

\begin{defn}\label{def:multipaths}
A \emph{multipath} of a digraph~$  G$ is a spanning subgraph such that each connected component is either a vertex or a simple path. The \emph{length} of a multipath~$H$, denoted by $\ell(  H)$, is the number of its edges. The set of multipaths of $G$ will be denoted by $\Mult(G)$, and the set of multipaths of length~$i$  by $\Mult_i(G)$.
\end{defn}
 
The set of multipaths of $  G$ has a natural partially ordered structure:

\begin{defn}\label{def:pathposet}
	The \emph{path poset} of $  G$ is the poset $(P(  G),<)$ associated to $  G$, that is, the set of multipaths of   G ordered by the relation of ``being a subgraph''. 
\end{defn}

With abuse of notation, we will also  write $P(  G)$ instead of  $(P(  G),<)$.
\begin{example}\label{ex:Pn}
Consider the coherently oriented linear graph $I_n$ with $n$ edges depicted in Figure~\ref{fig:nstep}. 
Then, $(P(I_n), < )$ is isomorphic to a Boolean poset.
Let $C_n$ be the coherently oriented polygonal graph with $n+1$ edges -- cf.~Figure~\ref{fig:poly}. Then, $(P(C_n) ,<)$, is isomorphic to a Boolean poset minus its maximum.
\begin{figure}[h]
	\begin{tikzpicture}[baseline=(current bounding box.center)]
		\tikzstyle{point}=[circle,thick,draw=black,fill=black,inner sep=0pt,minimum width=2pt,minimum height=2pt]
		\tikzstyle{arc}=[shorten >= 8pt,shorten <= 8pt,->, thick]
		
		\node[above] (v0) at (0,0) {$v_0$};
		\draw[fill] (0,0)  circle (.05);
		\node[above] (v1) at (1.5,0) {$v_1$};
		\draw[fill] (1.5,0)  circle (.05);
		\node[] at (3,0) {\dots};
		\node[above] (v4) at (4.5,0) {$v_{n-1}$};
		\draw[fill] (4.5,0)  circle (.05);
		\node[above] (v5) at (6,0) {$v_{n}$};
		\draw[fill] (6,0)  circle (.05);
		
		\draw[thick,  -latex] (0.15,0) -- (1.35,0);
		\draw[thick,  -latex] (1.65,0) -- (2.5,0);
		\draw[thick,  -latex] (3.4,0) -- (4.35,0);
		\draw[thick,  -latex] (4.65,0) -- (5.85,0);
	\end{tikzpicture}
	\caption{The coherently oriented linear graph $I_n$.}
	\label{fig:nstep}
\end{figure}
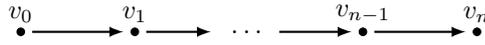

\begin{figure}[h]
\newdimen\R
\R=2.0cm
\begin{tikzpicture}
\draw[xshift=5.0\R, fill] (270:\R) circle(.05)  node[below] {$v_0$};
\draw[xshift=5.0\R,fill] (225:\R) circle(.05)   node[below left]   {$v_1$};
\draw[xshift=5.0\R,fill] (180:\R) circle(.05)   node[left] {$v_2$};
\draw[xshift=5.0\R,fill] (135:\R) circle(.05)   node[above left] {$v_3$};
\draw[xshift=5.0\R, fill] (90:\R) circle(.05)   node[above] {$v_4$};
\draw[xshift=5.0\R,fill] (45:\R) circle(.05)   node[above right] {$v_5$};
\draw[xshift=5.0\R,fill] (0:\R) circle(.05)  node[right] {$v_6$};
\draw[xshift=5.0\R,fill] (315:\R) circle(.05)   node[below right] {$v_n$};

\node[xshift=5.0\R] (v0) at (270:\R) { };
\node[xshift=5.0\R] (v1) at (225:\R) { };
\node[xshift=5.0\R] (v2) at (180:\R) { };
\node[xshift=5.0\R] (v3) at (135:\R) { };
\node[xshift=5.0\R] (v4) at (90:\R) { };
\node[xshift=5.0\R] (v5) at (45:\R) { };
\node[xshift=5.0\R] (v6) at (0:\R) { };
\node[xshift=5.0\R] (vn) at (315:\R) { };

\draw[thick,  -latex] (v0)--(v1);
\draw[thick,  -latex] (v1)--(v2);
\draw[thick,  -latex] (v2)--(v3);
\draw[thick,  -latex] (v3)--(v4);
\draw[thick,  -latex] (v4)--(v5);
\draw[thick,  -latex] (v5)--(v6);
\draw[thick,  -latex] (vn)--(v0);

\draw[xshift=5.0\R, fill] (292.5:\R);
\draw[xshift=5.0\R,fill] (247.5:\R);
\draw[xshift=5.0\R,fill] (157.5:\R);
\draw[xshift=5.0\R, fill] (112.5:\R);
\draw[xshift=5.0\R,fill] (67.5:\R);
\draw[xshift=5.0\R,fill] (22.5:\R);
\draw[xshift=4.95\R,fill] (337.5:\R)   node {$\cdot$} ;
\draw[xshift=4.95\R,fill] (333:\R)  node {$\cdot$} ;
\draw[xshift=4.95\R,fill] (342:\R)   node {$\cdot$} ;
\end{tikzpicture}
\caption{The coherently oriented cycle graph $C_n$.}
\label{fig:poly}
\end{figure}
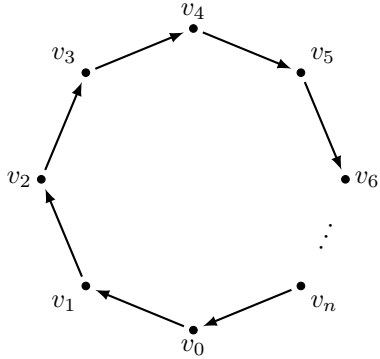
\end{example}
\subsection{Dynamical regions.}\label{sec:dynreg}
Let $G$ be a digraph, and let $G'\leq G$ be a subgraph. We will use the following terminology. The \emph{complement} $C_{G}(G')$ of  $G'$ in $G$ is the subgraph of $G$ spanned by the edges in $E(G)\setminus E(G')$. The \emph{boundary} $\partial_{G} G'$ of $G'$ in $G$, or simply $\partial G'$ when clear from the context, is defined as $\partial_{G} G'=V(G')\cap V(C_{G}(G))$. If $G$ ha at least an edge,  a vertex $v\in V(G)$ is called \emph{stable} if either the indegree or the outdegree of $v$ is zero, and \emph{unstable} otherwise.

\begin{defn}[{\cite[Definition~4.2]{terzo}}]\label{def:dynamicalregion}
Let $G$ be a digraph.
A \emph{dynamical region} in $G$ is a connected subgraph $R \leq G$, with at least one edge, such that: 
\begin{enumerate}[label = (\alph*)]
    \item\label{item:dandelion} all vertices in the boundary of $R$ are unstable in $G$, but stable in both $R$ and $C_{G}(R)$;
    \item\label{item:cicle} no edge of $R$ belongs to any oriented cycle in $G$ which is not contained in $R$.
\end{enumerate}
{A dynamical region is called \emph{stable} if all its non-boundary vertices are stable. Similarly, A dynamical region is called \emph{unstable} if all its non-boundary vertices are~unstable, and at least one vertex is unstable.}
\end{defn}

Observe that the non-empty intersection of two dynamical regions, say $R $ and $S$, still satisfies Items~\ref{item:dandelion} and \ref{item:cicle}. In particular, each connected component of $R \cap S$ is still a dynamical region.

\begin{defn}[{\cite[Definition~4.4]{terzo}}]
A \emph{dynamical module}, shortly a \emph{module}, of a digraph $G$ is a minimal dynamical region.
\end{defn}

Note that for each edge $e\in E(G)$ there exists a unique dynamical module of $G$ containing~$e$ -- cf.~\cite[Lemma~4.10]{terzo}. As a consequence, each directed graph has a unique (up to re-ordering) decomposition into dynamical modules -- cf.~\cite[Theorem~4.11]{terzo}.

\section{Multipath Matroids}   

We provide a complete classification of
digraphs $G$ for which the pair $(E(G), \Mult(G))$  defines a matroid -- where $\Mult(G)$ is the set of all multipaths in $G$. Then, we investigate some combinatorial properties of such matroids.
We start with some examples.

\begin{example}\label{ex:linear}
Let $I_n$ be the coherently oriented linear graph on $n+1$ vertices -- cf.~Figure~\ref{fig:nstep}. Then, every subset of $E(I_n)$ is a multipath, and   $(E(I_n), \Mult(I_n))$ is a matroid isomorphic to the matroid~$(E(G), \mc{P}(E(G)))$.
\end{example}
\begin{example}\label{ex:cycle}
Let $C_n$ be the coherently oriented cycle over $n$ vertices, cf.~Figure~\ref{fig:poly}. Then, the pair $(E(C_n), \Mult(C_n))$ is a matroid.
In fact, after labeling the edges of $C_n$ with the integers $\{1, \dots, n\}$, we can identify the ground set $E(C_n)$ with the set $X=\{1, \dots, n\}$ -- by associating to each edge the label given by its source vertex. Under this identification, the path poset $P(C_n)$ is isomorphic to~$\mc{P}(X) \setminus \{1, \dots, n\}$ and $(E(C_n), \Mult(C_n))$ is a matroid isomorphic to the uniform matroid~$U_{n-1, n}$.
\end{example}

Recall that a digraph $  G$ with $n$ edges is called a sink on $n+1$ vertices if it has a unique sink~$v$, no loops, and every edge of $G$ has $v$ as target.

\begin{example}\label{ex:sink}
Let $S_n$ be a sink 
on $n+1$ vertices. Then $(E(S_n), \Mult(S_n))$ is a matroid. 
As in Example~\ref{ex:cycle}, the ground set $E(S_n)$ can be identified with $X=\{1, \dots, n\}$, and, under this identification, an element $ m \in  \Mult(S_n) $ corresponds to a singleton. As a consequence, $(E(S_n), \Mult(S_n))$ is isomorphic to the uniform matroid $U_{1,n}$.
\end{example} 

\begin{example}\label{ex:loop} Let $G$ be the digraph with one vertex $v$ and $n$ loops at $v$. We have that $E(G)=n$ and $\Mult(G)=\emptyset$, then $(E(G), \Mult(G))$ is isomorphic to the uniform matroid $U_{0,n}$. Observe also that $U_{0,n}$ is isomorphic to the direct sum of $n$ copies of $U_{0,1}$.
\end{example}

So far, we showed some easy examples of pairs $(E(G),\Mult(G))$ yielding matroids.
Nevertheless, there are examples of digraphs $G$ for which $\Mult(G)$ is not the independent set of any matroid with ground set $E(G)$. Recall that a linear graph $A_n$ on $n+1$ vertices  is called \emph{alternating} if
whenever $(v_{i-1},v_i)\in E(A_n)$ for some~$i<n$, we have~$(v_{i+1},v_i)\in E(A_n)$ and, analogously, if $(v_{i},v_{i-1})\in E(A_n)$ then $(v_{i},v_{i+1})\in E(A_n)$.

\begin{example}\label{ex:NonMatroid} Consider the alternating graph $A_3$; then $(E(A_3), \Mult(A_3))$ is not a matroid. In fact, it is immediate to check (see Figure~\ref{fig:A3}) that in its path poset  there exist maximal elements with different cardinality; hence,  $\Mult(A_3)$ cannot be the set of independents of a matroid. 
Analogously, this happens for all alternating linear graphs on an even number of vertices.
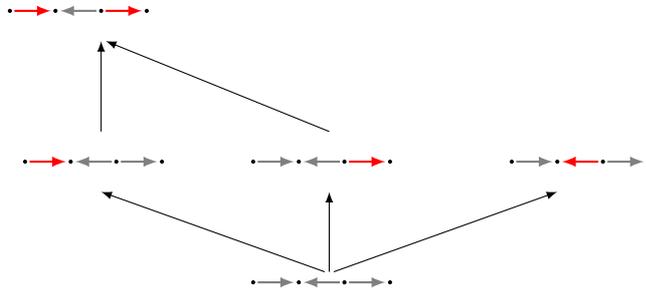
\begin{figure}[h]
	\begin{tikzpicture}[scale=0.4][baseline=(current bounding box.center)]
	\tikzstyle{point}=[circle,thick,draw=black,fill=black,inner sep=0pt,minimum width=2pt,minimum height=2pt]
	\tikzstyle{arc}=[shorten >= 8pt,shorten <= 8pt,->, thick]


	    \begin{scope}[shift={(-2.5,0)}]
			\node[above] (v0) at (0,0) {};
			\draw[fill] (0,0)  circle (.05);
			\node[above] (v1) at (1.5,0) {};
			\draw[fill] (1.5,0)  circle (.05);
			\node[above] (v2) at (3,0) {};
			\draw[fill] (3,0)  circle (.05);
			\node[above] (v3) at (4.5,0) {};
			\draw[fill] (4.5,0)  circle (.05);
			    \draw[thick, gray, -latex] (3.15,0) -- (4.35,0);
        	\draw[thick, gray, -latex] (0.15,0) -- (1.35,0);
         			\draw[thick, gray, -latex] (2.85,0) -- (1.65,0);
			
		\end{scope}

  	    \begin{scope}[shift={(-2.5,4)}]
			\node[above] (v0) at (0,0) {};
			\draw[fill] (0,0)  circle (.05);
			\node[above] (v1) at (1.5,0) {};
			\draw[fill] (1.5,0)  circle (.05);
			\node[above] (v2) at (3,0) {};
			\draw[fill] (3,0)  circle (.05);
			\node[above] (v3) at (4.5,0) {};
			\draw[fill] (4.5,0)  circle (.05);
			
			\draw[thick, red, -latex] (3.15,0) -- (4.35,0);
    	\draw[thick, gray, -latex] (0.15,0) -- (1.35,0);
     			\draw[thick, gray, -latex] (2.85,0) -- (1.65,0);
   
		\end{scope}

  	    \begin{scope}[shift={(6,4)}]
			\node[above] (v0) at (0,0) {};
			\draw[fill] (0,0)  circle (.05);
			\node[above] (v1) at (1.5,0) {};
			\draw[fill] (1.5,0)  circle (.05);
			\node[above] (v2) at (3,0) {};
			\draw[fill] (3,0)  circle (.05);
			\node[above] (v3) at (4.5,0) {};
			\draw[fill] (4.5,0)  circle (.05);
			
			\draw[thick, red, -latex] (2.85,0) -- (1.65,0);
       \draw[thick, gray, -latex] (3.15,0) -- (4.35,0);
        	\draw[thick, gray, -latex] (0.15,0) -- (1.35,0);
		\end{scope}
	
	    \begin{scope}[shift={(-10,4)}]
			\node[above] (v0) at (0,0) {};
			\draw[fill] (0,0)  circle (.05);
			\node[above] (v1) at (1.5,0) {};
			\draw[fill] (1.5,0)  circle (.05);
			\node[above] (v2) at (3,0) {};
			\draw[fill] (3,0)  circle (.05);
			\node[above] (v3) at (4.5,0) {};
			\draw[fill] (4.5,0)  circle (.05);
			
			\draw[thick, red, -latex] (0.15,0) -- (1.35,0);
      \draw[thick, gray, -latex] (3.15,0) -- (4.35,0);
      	\draw[thick, gray, -latex] (2.85,0) -- (1.65,0);
		\end{scope}

	    \begin{scope}[shift={(-10.5,9)}]
			\node[above] (v0) at (0,0) {};
			\draw[fill] (0,0)  circle (.05);
			\node[above] (v1) at (1.5,0) {};
			\draw[fill] (1.5,0)  circle (.05);
			\node[above] (v2) at (3,0) {};
			\draw[fill] (3,0)  circle (.05);
			\node[above] (v3) at (4.5,0) {};
			\draw[fill] (4.5,0)  circle (.05);
			
   \draw[thick, red, -latex] (3.15,0) -- (4.35,0);
   	\draw[thick, red, -latex] (0.15,0) -- (1.35,0);
    	\draw[thick, gray, -latex] (2.85,0) -- (1.65,0);
		\end{scope}

			\draw[ -latex,] (-0.15,0.35) -- (-7.5,3);
			\draw[ -latex] (0,0.35) -- (0,3);
   		\draw[ -latex,] (0.15,0.35) -- (7.5,3);
			\draw[ -latex] (-7.5,5) -- (-7.5,8);
			\draw[ -latex] (0,5) -- (-7.35,8);

\end{tikzpicture}
\caption{The path poset of the alternating graph $A_3$. The multipaths of $A_3$ are depicted in red.}\label{fig:A3} 
\end{figure}
\end{example}

We will use the following definition.

\begin{defn}\label{defn:multipathmatroid}
We say that a matroid $M$ is a \emph{multipath matroid} if there is a digraph  $G$ such that $M=(E(G), \Mult(G))$, in which case we denote it by $M_G$.
\end{defn}

Unlike  what happens with graphic matroids, trees with the same number of vertices do not necessary have isomorphic multipath matroids; in fact,
for $n > 1$, consider the sink $S_{n}$ (on $n+1$ vertices) and the  linear graph $I_{n}$ (on $n+1$ vertices, oriented as in Figure~\ref{fig:nstep}). 
Then, their multipath matroids are not isomorphic in view of Example~\ref{ex:linear} and Example~\ref{ex:sink}. 
\begin{rmk}\label{rmk:loopsareloops} Let $M_G$ be a multipath matroid. An edge $e=(v,w) \in E(G)$ is a loop for $M_G$ if and only if $v=w$, i.e.~if and only if $e$ is a loop in $G$. 
\end{rmk}

\subsection{Classification of Multipath Matroids}
As observed in Example~\ref{ex:NonMatroid}, the set of multipaths of a given digraph does not always satisfies the axioms of a family of independent sets. In this section we completely classify  digraphs with this property. A first step is to identify a possible set of circuits in $M_G$, i.e.~to identify the minimal spanning subgraphs of $G$ that are not multipaths.

\begin{proposition}\label{prop:minimal} Let $G$ be a digraph. Consider a (possibly disconnected) spanning subgraph $G'$ such that $G' \notin \Mult(G)$.
Then, $G'$ is  minimal by inclusion if, and only if, it is (the disjoint union of some vertices with) either (1) a linear sink or source, or (2) a coherently oriented cycle (possibly a loop).
\end{proposition}
\proof 

First, it is clear from Definition~\ref{def:multipaths} that if $G'$ 
satisfies~(1) or (2), then $G'$ is a minimal spanning subgraph of $G$ such that 
$G' \notin \Mult(G)$.

Assume that  $G'$ is minimal.
Recall that a connected graph with every vertex of valence  smaller than or equal to $2$ is either a cycle or a linear graph. 

If $G'$ contains a linear source or sink, or a loop, the thesis follows by minimality. Otherwise, every vertex of $G'$ has valence at most $2$. Since $G'$ is not a multipath, it must contain an coherently oriented cycle (possibly of length $2$). Then, by minimality this must be the only connected component which is not a vertex. 
\endproof

The following corollary is a direct consequence of Proposition~\ref{prop:minimal};

\begin{corollary}\label{cor:circ}
Let $G$ be a digraph and suppose that $M_G=(E(G), \Mult(G))$ is a matroid. Then,  the set of circuits of $M_G$  can be identified with all subgraphs that are linear sinks, linear sources, or coherently oriented cycles (including loops). 
\end{corollary}
\begin{rmk}\label{rmk:circ}
As a consequence of Corollary~\ref{cor:circ}, if $G'$ is a subgraph of $G$ and $(E(G), \Mult(G))$ is a multipath matroid, then $(E(G'), \Mult(G'))$ is again a multipath matroid. This can be seen by restriction: the set of linear sinks, linear sources and coherently oriented cycles of $G'$ is the subset of the circuits of $M_G$ whose edges belong to $G'$. Thus, conditions (C1)--(C3) are satisfied. In particular, if $G'$ is a subgraph of $G$ and $(E(G'), \Mult(G'))$ is not a multipath matroid, then $(E(G), \Mult(G))$ is not a multipath matroid. 
\end{rmk}
By definition, multipaths are not affected by  the existence of loops in a digraph. However, multipath matroids are sensitive to them. In fact, we have the following:

\begin{rmk}\label{rmk:looped} Let  $G$ be a graph with $k$ loops, and denote by $\widetilde{G}$ the graph obtained from $G$ by removing all loops. Then, we have that $\Mult(G)=\Mult(\widetilde{G})$ and in particular $M_G$ is a matroid if and only if $M_{\widetilde{G}}$ is a matroid. Observe that $M_G$ and $M_{\widetilde{G}}$ are not isomorphic -- since their ground sets are distinct. More precisely, $M_G \cong M_{\widetilde{G}} \oplus U_{0,k} \cong M_{\widetilde{G}} \oplus U_{0,1}^{\oplus k}.$
\end{rmk}

We are ready to state the main result of the section.

\begin{theorem}\label{thm:boom}Let $G$ be a digraph. Then, the pair $(E(G), \Mult(G))$ is a matroid if and only if the following two condition are satisfied:
\begin{enumerate}[label={(MP\arabic*)}]
\item \label{MP1} $G$ does not contain (up to edge reversing) one of the digraphs in Figure~\ref{fig:avoiding} as subgraphs;
\item\label{MP2} every coherently oriented cycle 
(loops excluded)
in $G$ is a connected component.
\end{enumerate}
\end{theorem}
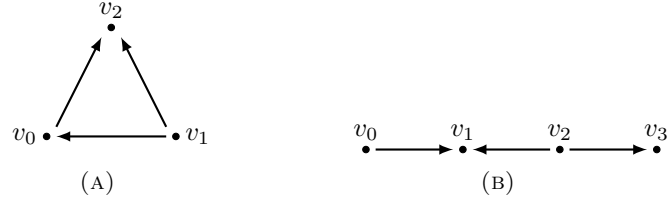
\begin{figure}[h]
	\centering
	\begin{subfigure}[b]{0.3\textwidth}
		\centering
	\begin{tikzpicture}[baseline=(current bounding box.center),scale =.85]
		\tikzstyle{point}=[circle,thick,draw=black,fill=black,inner sep=0pt,minimum width=2pt,minimum height=2pt]
		\tikzstyle{arc}=[shorten >= 8pt,shorten <= 8pt,->, thick]
		
		\node[left] (v0) at (0,0) {$v_0$};
		\draw[fill] (0,0)  circle (.05);
		\node[right] (v1) at (2,0) {$v_1$};
		\draw[fill] (2,0)  circle (.05);
		\node[above] (v2) at (1,1.7) {$v_{2}$};
		\draw[fill] (1,1.7)  circle (.05);

		\draw[thick,  -latex]  (1.85,0) -- (0.15,0);
		\draw[thick,   -latex]  (0.15,0.15) to  (0.85,1.55) ;
    	\draw[thick,   -latex] (1.85,0.15) -- (1.15,1.55);
	\end{tikzpicture}
		\caption{\phantom{A}}\label{subfig:avoidingA}
	\end{subfigure}
	\begin{subfigure}[b]{0.4\textwidth}
	\centering
	\begin{tikzpicture}[baseline=(current bounding box.center),scale =.85]
		\tikzstyle{point}=[circle,thick,draw=black,fill=black,inner sep=0pt,minimum width=2pt,minimum height=2pt]
		\tikzstyle{arc}=[shorten >= 8pt,shorten <= 8pt,->, thick]
		
		\node[above] (v0) at (0,0) {$v_0$};
		\draw[fill] (0,0)  circle (.05);
		\node[above] (v) at (1.5,0) {$v_1$};
		\draw[fill] (1.5,0)  circle (.05);
		\node[above] (v1) at (3,0) {$v_{2}$};
		\draw[fill] (3,0)  circle (.05);
		 \node[above] (v3) at (4.5,0) {$v_{3}$};
		 \draw[fill] (4.5,0)  circle (.05);
		
		\draw[thick,   -latex] (0.15,0) -- (1.35,0);
		\draw[thick,   -latex]  (2.85,0) to (1.65,0) ;
  	\draw[thick,   -latex] (3.15,0) -- (4.35,0);
	\end{tikzpicture}
				\caption{\phantom{B}}\label{subfig:avoidingB}
\end{subfigure}
	\caption{Digraphs to be avoided to get a multipath matroid.}
	\label{fig:avoiding}
\end{figure}
To prove Theorem~\ref{thm:boom} we need a preliminary result. 
\begin{lemma}\label{lem:C2}
Let $G$ be a digraph without loops.
Let $A, B \notin \Mult(G)$ be minimal spanning subgraphs of $G$ such that $E(A) \cap E(B) \neq \emptyset $. Suppose that $A$ contains a linear sink or source, that  $B$ contains a coherently oriented cycle, and that the cycle in $B$ is the unique coherently oriented cycle in $A \cup B$. Then there exists $x \in E(A) \cap E(B)$ such that $A \cup B \setminus \{x\} $ is a multipath. 
\end{lemma}
\proof  The digraph $A \cup B$ contains one of the subgraphs displayed in Figures~\ref{fig:AuB} and~\ref{fig:AuBdigon}. Observe that $A \cap B$ consist of a single edge $e$. This edge -- depicted in red in Figures~\ref{fig:AuB} and~\ref{fig:AuBdigon}, for the sake of visualization -- if removed from $A \cup B \setminus \{e\} $ yields a coherently oriented linear graph. \endproof
\begin{figure}[h]
	\centering
	\begin{subfigure}[b]{0.3\textwidth}
		\centering
	\begin{tikzpicture}[baseline=(current bounding box.center),scale =.85]
		\tikzstyle{point}=[circle,thick,draw=black,fill=black,inner sep=0pt,minimum width=2pt,minimum height=2pt]
		\tikzstyle{arc}=[shorten >= 8pt,shorten <= 8pt,->, thick]
		
		\node[above] (v0) at (0,0) {$v_0$};
		\draw[fill] (0,0)  circle (.05);
		\node[above] (v1) at (1.5,0) {$v$};
		\draw[fill] (1.5,0)  circle (.05);
		\node[above] (v2) at (3,1) {$v_{1}$};
		\draw[fill] (3,1)  circle (.05);
		\node[above] (v3) at (3,-1) {$v_{2}$};
		\draw[fill] (3,-1)  circle (.05);
		
		\draw[thick, red,  -latex]  (0.15,0) -- (1.35,0) ;
		\draw[thick, gray,  -latex]  (2.85,0.95) -- (1.65,0.05) ;
		\draw[thick, gray, -latex] (1.65,-0.05) -- (2.85,-0.95);
	\end{tikzpicture}
				\caption{\phantom{A}}
\end{subfigure}
	\begin{subfigure}[b]{0.3\textwidth}
	\centering
	\begin{tikzpicture}[baseline=(current bounding box.center),scale =.85]
		\tikzstyle{point}=[circle,thick,draw=black,fill=black,inner sep=0pt,minimum width=2pt,minimum height=2pt]
		\tikzstyle{arc}=[shorten >= 8pt,shorten <= 8pt,->, thick]
		
		\node[above] (v0) at (0,0) {$v_0$};
		\draw[fill] (0,0)  circle (.05);
		\node[above] (v1) at (1.5,0) {$v$};
		\draw[fill] (1.5,0)  circle (.05);
		\node[above] (v2) at (3,1) {$v_{1}$};
		\draw[fill] (3,1)  circle (.05);
		\node[above] (v3) at (3,-1) {$v_{2}$};
		\draw[fill] (3,-1)  circle (.05);
		
		\draw[thick, red, -latex](1.35,0) --  (0.15,0)  ;
		\draw[thick, gray, -latex] (1.65,0.05) -- (2.85,0.95) ;
		\draw[thick, gray, -latex] (2.85,-0.95) -- (1.65,-0.05);
	\end{tikzpicture}
				\caption{ \phantom{B }}
\end{subfigure}
	\caption{The possible (local) configurations of union of a linear circuit with a coherently oriented cycle of length greater than 2.}
	\label{fig:AuB}
\end{figure}
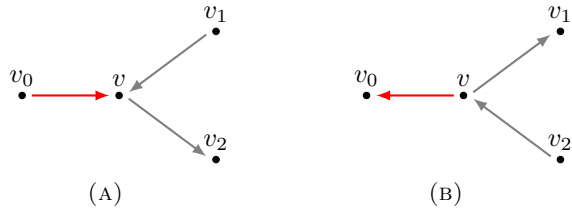

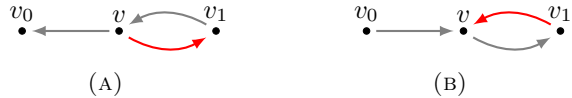
\begin{figure}[h]
	\centering
	\begin{subfigure}[b]{0.3\textwidth}
		\centering
	\begin{tikzpicture}[baseline=(current bounding box.center),scale =.85]
		\tikzstyle{point}=[circle,thick,draw=black,fill=black,inner sep=0pt,minimum width=2pt,minimum height=2pt]
		\tikzstyle{arc}=[shorten >= 8pt,shorten <= 8pt,->, thick]
		
		\node[above] (v0) at (0,0) {$v_0$};
		\draw[fill] (0,0)  circle (.05);
		\node[above] (v1) at (1.5,0) {$v$};
		\draw[fill] (1.5,0)  circle (.05);
		\node[above] (v2) at (3,0) {$v_{1}$};
		\draw[fill] (3,0)  circle (.05);

		\draw[thick, gray,  -latex]  (1.35,0) -- (0.15,0) ;
		\draw[thick, red,  -latex]  (1.65,-0.1) to [bend right]  (2.85,-0.1) ;
		\draw[thick, gray, -latex]   (2.85,0.1) to [bend right] (1.65,0.1) ;
	\end{tikzpicture}
		\caption{\phantom{A}}
	\end{subfigure}
	\begin{subfigure}[b]{0.3\textwidth}
	\centering
	\begin{tikzpicture}[baseline=(current bounding box.center),scale =.85]
		\tikzstyle{point}=[circle,thick,draw=black,fill=black,inner sep=0pt,minimum width=2pt,minimum height=2pt]
		\tikzstyle{arc}=[shorten >= 8pt,shorten <= 8pt,->, thick]
		
		\node[above] (v0) at (0,0) {$v_0$};
		\draw[fill] (0,0)  circle (.05);
		\node[above] (v) at (1.5,0) {$v$};
		\draw[fill] (1.5,0)  circle (.05);
		\node[above] (v1) at (3,0) {$v_{1}$};
		\draw[fill] (3,0)  circle (.05);
		
		\draw[thick, gray,   -latex] (0.15,0) -- (1.35,0);
		\draw[thick, gray,  -latex]  (1.65,-0.1) to [bend right]  (2.85,-0.1) ;
		\draw[thick, red,  -latex]  (2.85,0.1) to [bend right] (1.65,0.1) ;
	\end{tikzpicture}
				\caption{\phantom{B}}
\end{subfigure}
	\caption{The possible (local) configurations of union of a linear circuit with a coherently oriented cycle of length 2.}
	\label{fig:AuBdigon}
\end{figure}
Now we are ready to prove Theorem~\ref{thm:boom}.
\begin{proof}[Proof of Theorem~\ref{thm:boom}]
In this proof we will use the definition of matroid in terms of circuits -- see  Remark~\ref{defn:matroidcircuits}. By Remark~\ref{rmk:looped}, we can assume that $G$ is without loops.

Suppose first that $M_G=(E(G), \Mult(G))$ is a matroid. Since the pair $(E(A_3),\Mult(A_3))$ is not a matroid, Remark~\ref{rmk:circ} implies that $G$ cannot contain the digraph in Figure~\ref{fig:avoiding}.(B). Analogously,~$G$ cannot contain the graph in Figure~\ref{fig:avoiding}.(A). Then, Condition~\ref{MP1} follows. Now, we turn to the proof of~\ref{MP2}.
By Corollary~\ref{cor:circ}, the set of circuits of $M_G$ can be identified with the set of spanning subgraphs of $G$ that are  linear sinks, linear sources or coherently oriented cycles.  
Let $A$ and $B$ be circuits of $M_G$.
Condition~\ref{C3} in the definition of matroid is satisfied if, and only if, for every $x \in E(A \cap B)$ the subgraph $A \cup B \setminus \{x\}$ is not a multipath.
Note that the intersection of two distinct coherently oriented cycles of length two is empty, cf.~Remark~\ref{rem:digons}. 
Let  $C$ be a coherently oriented cycle of $G$ whose length is greater than~$2$. We want to prove that $C$ cannot intersect any linear sink or source in $G$. Observe that either we are in the hypotheses of Lemma~\ref{lem:C2} -- and thus we are done, or we have a subgraph of the form illustrated in Figure~\ref{fig:AuBdigon}. In this latter case, Remark~\ref{rmk:circ} yields a contradiction; in fact, the multipaths of either graph in Figure~\ref{fig:AuBdigon} do not form a matroid. 

Now, we want to analyse the intersections between $C$ and another coherently oriented cycle~$C'$ (possibly of length $2$). 
The connected component of $C \cup C'$ which is not a vertex must contain one of the digraphs in Figure~\ref{fig:Cyrcdigon}.
In all such cases, there exists a linear sink or a linear source (highlighted in red Figure~\ref{fig:Cyrcdigon}) in $G$ that intersects $C$. This is a contradiction by~ Remark~\ref{rmk:circ}, and Condition~\ref{MP2} follows.
 
\begin{figure}[h]
	\centering
	\begin{subfigure}[b]{0.3\textwidth}
		\centering
	\begin{tikzpicture}[baseline=(current bounding box.center),scale =.85]
		\tikzstyle{point}=[circle,thick,draw=black,fill=black,inner sep=0pt,minimum width=2pt,minimum height=2pt]
		\tikzstyle{arc}=[shorten >= 8pt,shorten <= 8pt,->, thick]
		
		\node[above] (v0) at (0,0) {$v_0$};
		\draw[fill] (0,0)  circle (.05);
		\node[above] (v1) at (1.5,0) {$v$};
		\draw[fill] (1.5,0)  circle (.05);
		\node[above] (v2) at (3,1) {$v_{1}$};
		\draw[fill] (3,1)  circle (.05);
		\node[above] (v3) at (3,-1) {$v_{2}$};
		\draw[fill] (3,-1)  circle (.05);
		
		\draw[thick, gray,  -latex] (0.15,0) -- (1.35,0);
		\draw[thick, red, -latex] (1.65,0.05) -- (2.85,0.95);
		\draw[thick, red,  -latex] (1.65,-0.05) -- (2.85,-0.95);
	\end{tikzpicture}
		\caption{\phantom{A}}
	\end{subfigure}
	\begin{subfigure}[b]{0.3\textwidth}
	\centering
	\begin{tikzpicture}[baseline=(current bounding box.center),scale =.85]
		\tikzstyle{point}=[circle,thick,draw=black,fill=black,inner sep=0pt,minimum width=2pt,minimum height=2pt]
		\tikzstyle{arc}=[shorten >= 8pt,shorten <= 8pt,->, thick]
		
		\node[above] (v0) at (0,0) {$v_0$};
		\draw[fill] (0,0)  circle (.05);
		\node[above] (v1) at (1.5,0) {$v$};
		\draw[fill] (1.5,0)  circle (.05);
		\node[above] (v2) at (3,1) {$v_{1}$};
		\draw[fill] (3,1)  circle (.05);
		\node[above] (v3) at (3,-1) {$v_{2}$};
		\draw[fill] (3,-1)  circle (.05);
		
		\draw[thick, gray, -latex] (1.35,0) -- (0.15,0) ;
		\draw[thick, red,  -latex] (2.85,0.95) -- (1.65,0.05);
		\draw[thick, red,  -latex]  (2.85,-0.95) -- (1.65,-0.05) ;
	\end{tikzpicture}
				\caption{\phantom{B}}
\end{subfigure}
~\\\vspace{2em}
	\begin{subfigure}[b]{0.3\textwidth}
		\centering
	\begin{tikzpicture}[baseline=(current bounding box.center),scale =.85]
		\tikzstyle{point}=[circle,thick,draw=black,fill=black,inner sep=0pt,minimum width=2pt,minimum height=2pt]
		\tikzstyle{arc}=[shorten >= 8pt,shorten <= 8pt,->, thick]
		
		\node[left] (v0) at (0,0) {$v_0$};
		\draw[fill] (0,0)  circle (.05);
		\node[right] (v1) at (2,0) {$v_1$};
		\draw[fill] (2,0)  circle (.05);
		\node[above] (v2) at (1,1.7) {$v_{2}$};
		\draw[fill] (1,1.7)  circle (.05);

		\draw[thick, gray,  -latex]  (0.15,-0.05) to [bend right] (1.85,-0.05);
  		\draw[thick, red,   -latex]  (1.85,0.05) to [bend right] (0.15,0.05);
		\draw[thick,  gray,  -latex] (0.85,1.55) to (0.15,0.15);
    	\draw[thick, red,  -latex] (1.85,0.15) -- (1.15,1.55) ;
	\end{tikzpicture}
	 \caption{\phantom{A}}
	\end{subfigure}
	\begin{subfigure}[b]{0.4\textwidth}
	\centering
	\begin{tikzpicture}[baseline=(current bounding box.center),scale =.85]
		\tikzstyle{point}=[circle,thick,draw=black,fill=black,inner sep=0pt,minimum width=2pt,minimum height=2pt]
		\tikzstyle{arc}=[shorten >= 8pt,shorten <= 8pt,->, thick]
		
		\node[above] (v0) at (0,0) {$v_0$};
		\draw[fill] (0,0)  circle (.05);
		\node[above] (v) at (1.5,0) {$v$};
		\draw[fill] (1.5,0)  circle (.05);
		\node[above] (v1) at (3,0) {$v_{1}$};
		\draw[fill] (3,0)  circle (.05);
		 \node[above] (v3) at (4.5,0) {$v_{2}$};
		 \draw[fill] (4.5,0)  circle (.05);
		
		\draw[thick, red,   -latex] (0.15,0) -- (1.35,0);
		\draw[thick, gray, -latex]  (1.65,-0.05) to [bend right]  (2.85,-0.05) ;
		\draw[thick, red,   -latex]  (2.85,0.05) to [bend right] (1.65,0.05) ;
  	\draw[thick, gray,  -latex] (3.15,0) -- (4.35,0);
	\end{tikzpicture}
	 \caption{\phantom{A}}
\end{subfigure}
	\caption{The (local) configurations near the intersections of two coherently oriented cycles.}
	\label{fig:Cyrcdigon}
\end{figure}
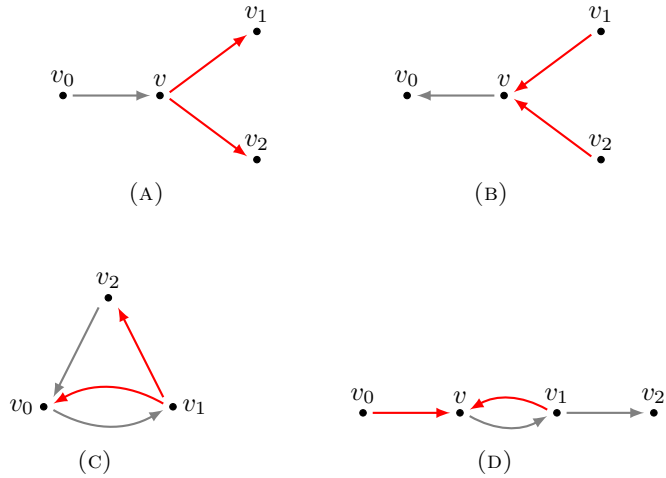
\emph{Vice versa}, we want to prove  that, if $G$ satisfies Conditions~\ref{MP1} and~\ref{MP2}, then the pair $M_G=(E(G), \Mult(G))$ is a matroid.  
By definition of multipaths, every circuit of $M_G$ must have at least one edge -- thus, \ref{C1} follows. 
Proposition~\ref{prop:minimal} identifies the possible circuits in~$M_G$, and also Condition~\ref{C2} immediately follows from this description. 
It remains to check Condition~\ref{C3} of Remark~\ref{defn:matroidcircuits}. 
By Condition~\ref{MP2}, the set of edges of any coherently oriented cycle is disjoint from the set of edges of any other circuit. Consequently, we need to check that Condition~\ref{C3} is satisfied for circuits $A$ and $B$  both containing a linear sink or source and having a common edge. Observe that by Condition~\ref{MP1}, the set of edges of the graph $A \cup B$ cannot be arranged as in one of the digraphs in Figure~\ref{fig:avoiding}. Therefore, $A \cup B$ must be, up to the disjoint union with some vertices, one of the graphs displayed in Figure~\ref{fig:usink}, where the edge in $A \cap B$ is highlighted in red. In both cases, the subgraph spanned by the vertices $\{ v, v_1, v_2\}$ (in gray in~Figure~\ref{fig:usink}) is a circuit, hence concluding the proof.
\begin{figure}[h]
	\centering
	\begin{subfigure}[b]{0.3\textwidth}
		\centering
	\begin{tikzpicture}[baseline=(current bounding box.center),scale =.85]
		\tikzstyle{point}=[circle,thick,draw=black,fill=black,inner sep=0pt,minimum width=2pt,minimum height=2pt]
		\tikzstyle{arc}=[shorten >= 8pt,shorten <= 8pt,->, thick]
		
		\node[above] (v0) at (0,0) {$v_0$};
		\draw[fill] (0,0)  circle (.05);
		\node[above] (v1) at (1.5,0) {$v$};
		\draw[fill] (1.5,0)  circle (.05);
		\node[above] (v2) at (3,1) {$v_{1}$};
		\draw[fill] (3,1)  circle (.05);
		\node[above] (v3) at (3,-1) {$v_{2}$};
		\draw[fill] (3,-1)  circle (.05);
		
		\draw[thick, red,  -latex]  (1.35,0) -- (0.15,0) ;
		\draw[thick, gray, -latex] (1.65,0.05) -- (2.85,0.95);
		\draw[thick, gray, -latex] (1.65,-0.05) -- (2.85,-0.95);
	\end{tikzpicture}
		\caption{\phantom{A}}
	\end{subfigure}
	\begin{subfigure}[b]{0.3\textwidth}
	\centering
	\begin{tikzpicture}[baseline=(current bounding box.center),scale =.85]
		\tikzstyle{point}=[circle,thick,draw=black,fill=black,inner sep=0pt,minimum width=2pt,minimum height=2pt]
		\tikzstyle{arc}=[shorten >= 8pt,shorten <= 8pt,->, thick]
		
		\node[above] (v0) at (0,0) {$v_0$};
		\draw[fill] (0,0)  circle (.05);
		\node[above] (v1) at (1.5,0) {$v$};
		\draw[fill] (1.5,0)  circle (.05);
		\node[above] (v2) at (3,1) {$v_{1}$};
		\draw[fill] (3,1)  circle (.05);
		\node[above] (v3) at (3,-1) {$v_{2}$};
		\draw[fill] (3,-1)  circle (.05);
		
		\draw[thick, red, -latex]  (0.15,0) -- (1.35,0) ;
		\draw[thick, gray, -latex] (2.85,0.95) -- (1.65,0.05);
		\draw[thick, gray, -latex]  (2.85,-0.95) -- (1.65,-0.05) ;
	\end{tikzpicture}
				\caption{\phantom{B}}
\end{subfigure}
	\caption{The union of two linear sources (A) and of two linear sinks (B).}
	\label{fig:usink}
\end{figure}
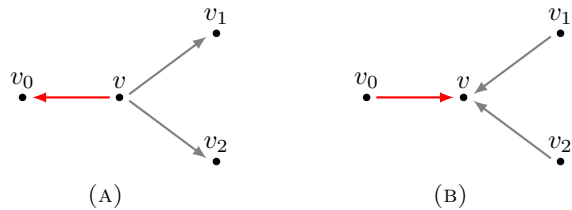
\end{proof}

In view of Theorem~\ref{thm:boom}, we give the following definition.

\begin{defn}
    A digraph $G$ is called an \emph{MP-digraph} if it satisfies Conditions~\ref{MP1} and \ref{MP2} of Theorem~\ref{thm:boom}.
\end{defn}

We provide some (non-)examples of MP-digraphs. Recall that a transitive tournament is a directed graph on vertices $0,\dots , n$ with
edges $(i, j)$ for all $i \leq  j$.

\begin{example} If $T_n$ is a transitive tournament on $n>2$ vertices,  then $M_{T_n}$ is not a multipath matroid. In fact, every transitive tournament with more than $3$ vertices contains a subgraph isomorphic to the digraph in Figure~\ref{fig:avoiding}.(A).  
\end{example}
\begin{example}
Consider the digraph in Figure~\ref{fig:noeqioriented}. Observe that the underlying unoriented graph has a cycle which is not a connected component. Nevertheless this cycle is not coherently oriented, and the graph has no subgraphs isomorphic to the ones in Figure~\ref{fig:avoiding}. Consequently $M_G$ is a multipath matroid. 
\begin{figure}[h]
    \centering
    \begin{tikzpicture}[thick,scale = 1.5]
        \node (a1) at (1,0) {};
        \node (a2) at (2,0) {};
        \node (b1) at (1,1) {};
        \node (b2) at (2,1) {};
        \node (c2) at (3,-0.5) {};
        
        \node[below left]  at (1,0) {};
        \node[below right] at (2,0) {};
        \node[above] at (1,1) {};
        \node[above] at (2,1) {};
        
        \draw[fill] (a1) circle (0.03);
        \draw[fill] (b1) circle (0.03);
        \draw[fill] (a2) circle (0.03);
        \draw[fill] (b2) circle (0.03);
        \draw[fill] (c2) circle (0.03);

        \draw[-latex,] (a1) -- (a2);
        \draw[-latex, ] (b1) -- (b2);
        \draw[-latex, ] (b2) -- (a2);
        \draw[-latex, ] (b1) -- (a1);
        \draw[-latex, ] (a2) -- (c2);
    \end{tikzpicture}
    \caption{A digraph with a (non-coherently oriented) cycle that is not a connected component.}
    \label{fig:noeqioriented}
\end{figure}
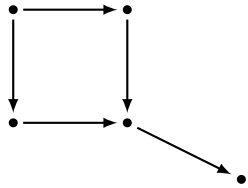
\end{example}

As an immediate consequence of  Remark~\ref{rmk:circ}, the following theorem holds: 
\begin{theorem}\label{thm:components} Let $G$ be an  MP-digraph and let $G_1 , \dots, G_k$ be its connected components. Then, 
\[M_G=M_{G_1} \oplus \dots \oplus M_{G_k}.\]
\end{theorem}
\proof Without loss of generality, we can assume  $k=2$. In this case, we have that $E(G)= E(G_1) \sqcup E(G_2)$. Moreover, a multipath for $G$ is the union of a multipath for $G_1$ and of a multipath for $G_2$; in particular, $\Mult(G)=\Mult(G_1) \sqcup \Mult(G_2)$. By Remark~\ref{rmk:circ} both $M_{G_1}$ and $M_{G_2}$ are multipath matroids. Then, by Definition~\ref{defn:summat}, $M_G=M_{G_1}\oplus M_{G_2}$. \endproof

Observe that if $G'$ is obtained from $G$ by edge reversing, i.e.~by reversing the orientation of all edges, then the matroids $M_G$ and $M_{G'}$ are isomorphic. It is then natural to ask if the multipath matroid uniquely determines  a directed graph $G$ up to edge reversing. The answer to this question is negative. In fact, the graphs in Figure~\ref{fig:isoM} have  isomorphic path posets and isomorphic multipath matroids. Note that their dynamical modules are isomorphic. This is not always the case, take for instance $A_2$ and the cycle of length $2$.

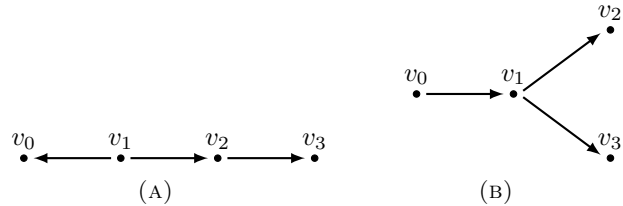
\begin{figure}[h]
	\centering
	\begin{subfigure}[b]{0.3\textwidth}
		\centering
	\begin{tikzpicture} [baseline=(current bounding box.center),scale =.85]
		\tikzstyle{point}=[circle,thick,draw=black,fill=black,inner sep=0pt,minimum width=2pt,minimum height=2pt]
		\tikzstyle{arc}=[shorten >= 8pt,shorten <= 8pt,->, thick]
		
		\node[above] (v0) at (0,0) {$v_0$};
		\draw[fill] (0,0)  circle (.05);
		\node[above] (v1) at (1.5,0) {$v_1$};
		\draw[fill] (1.5,0)  circle (.05);
		\node[above] (v2) at (3,0) {$v_{2}$};
		\draw[fill] (3,0)  circle (.05);
		\node[above] (v3) at (4.5,0) {$v_{3}$};
		\draw[fill] (4.5,0)  circle (.05);
		
		\draw[thick,  -latex] (1.35,0) -- (0.15,0);
		\draw[thick,  -latex] (1.65,0) -- (2.85,0);
		\draw[thick,  -latex] (3.15,0) -- (4.35,0);
	\end{tikzpicture}
		\caption{\phantom{A}}\label{fig:A}
	\end{subfigure}
	\begin{subfigure}[b]{0.3\textwidth}
	\centering
	\begin{tikzpicture} [baseline=(current bounding box.center),scale =.85]
		\tikzstyle{point}=[circle,thick,draw=black,fill=black,inner sep=0pt,minimum width=2pt,minimum height=2pt]
		\tikzstyle{arc}=[shorten >= 8pt,shorten <= 8pt,->, thick]
		
		\node[above] (v0) at (0,0) {$v_0$};
		\draw[fill] (0,0)  circle (.05);
		\node[above] (v1) at (1.5,0) {$v_1$};
		\draw[fill] (1.5,0)  circle (.05);
		\node[above] (v2) at (3,1) {$v_{2}$};
		\draw[fill] (3,1)  circle (.05);
		\node[above] (v3) at (3,-1) {$v_{3}$};
		\draw[fill] (3,-1)  circle (.05);
		
		\draw[thick,  -latex] (0.15,0) -- (1.35,0)  ;
		\draw[thick,  -latex] (1.65,0.05) -- (2.85,0.95);
		\draw[thick,  -latex] (1.65,-0.05) -- (2.85,-0.95) ;
	\end{tikzpicture}
				\caption{ \phantom{C }}\label{fig:Y}
\end{subfigure}
	\caption{Two digraph with isomorphic multipath matroids}
	\label{fig:isoM}
\end{figure}

\subsection{Multipath Matroids are graphic.}
The aim of this subsection is to prove some general properties of multipath matroids. In particular, we compare them to graphic matroids. 

We start by studying the representability of multipath matroids. Recall that a matroid is binary if it is representable over the finite field $\mathbb{F}_2$, and regular if it is representable over any field.

\begin{theorem}
Let $G$ be an MP-digraph. Then $M_G$ is a binary matroid.
\end{theorem}
\proof By \cite[Chapter 10]{welshmath}, a matroid $M$ is binary if and only if for every pair of circuits of $M$, their symmetric difference contains another circuit.  
In view of Theorem~\ref{thm:boom}, the edges of two circuits in $M_G$ are either disjoint or are arranged as in Figure~\ref{fig:usink}.
In the former case, their  symmetric difference  is equal to their union; hence, it contains a circuit. 
If the edges of the two circuits are as in Figure~\ref{fig:usink}, the edges in the symmetric difference of the two circuits are those depicted in gray -- in particular, they are the edges of a circuit. Hence, 
$M_G$ is binary.
\endproof

Recall that a graphic matroid is a matroid which is isomorphic to $(E(G'), \mc{I}_{G'} )$, for some (unoriented) graph $G'$, where $\mc{I}_{G'}$ is the set of subforests in $G'$ (including the empty set). 
Every graphic matroid is binary. Then, a natural question is  whether  multipath matroids are also graphic. Aiming to give a complete answer to this question, we need some preliminary results connecting the  matroid structure of $M_G$ with  properties of the dynamical regions of $G$ -- see Section~\ref{sec:dynreg}. We start with the following theorem, which is a rephrasing of \cite[Theorem 4.11]{terzo} in terms of path posets:
\begin{theorem}\label{thm:modules}  Let $G_1 , \dots, G_k$ be dynamical modules for a connected digraph $G$. Then, we have the decomposition
\[P_G = P_{G_1} \times \dots  \times P_{G_k} \]
of the path poset $P_G$ as product of path posets.
\end{theorem}
Note that, in particular, the path poset of $G$ is isomorphic to the path poset of $G_1 \sqcup  \dots \sqcup G_k$.
 \begin{corollary}\label{thm:modulesmatroid} Let $G$ be a connected MP-digraph and $G_1 , \dots, G_k$ denote its dynamical modules.  Then, we have 
$M_G=M_{G_1} \oplus \dots \oplus M_{G_k}$.
\end{corollary}

\proof By definition of dynamical region, we have that $E(G)= E(G_1) \sqcup \dots \sqcup E(G_k)$. Because of Theorem~\ref{thm:modules}, we also have that $\Mult(G)=\Mult(G_1) \sqcup \dots \sqcup \Mult(G_k)$. Then, the statement follows in view of Definition~\ref{defn:summat} and  Remark~\ref{rmk:circ}. \endproof

In the next theorem we analyse better what are the possible dynamical regions appearing in a digraph that satisfies the conditions of Theorem~\ref{thm:boom}.

\begin{proposition}\label{prop:modules}
Let $G$ be a connected loopless MP-digraph that is not a coherently oriented cycle. Suppose that $G$  has a unique dynamical module. Then $G$ is a sink or a source. 
\end{proposition}
\proof 
Since $G$ is connected,  by Condition~\ref{MP2} in Theorem~\ref{thm:boom} it cannot contain coherently oriented cycles.

Assume $E(G) > 2$, and choose a vertex $v \in V(G)$, the case $E(G)\leq 2$ being straightforward.
Suppose first that the valence of every vertex of $G$ is smaller than or equal to $2$. Then, the undirected graph underlying $G$ is a linear graph or a cycle.
If the underlying graph is a cycle, then it has a unique dynamical region if it is alternating or coherently oriented. Therefore $G$ must be alternating.  This yields a contradiction, since it does not contain the digraphs in Figure~\ref{fig:avoiding}. It follows that $G$ is a linear graph. The unique linear graphs satisfying both  conditions~\ref{MP1} and \ref{MP2}, and having a unique dynamical module, are the linear sources  and the linear sinks. Proving the statement in this case.

Therefore, we can now assume that there is a vertex $v$ of valence greater than $2$.
Denote by~$I(v)$ the set of incoming edges  at $v$. Since the valence of $v$ is greater than $2$, up to edge reversing, we can assume $|I(v)|  \geq 2$. Let $S$ be the connected subgraph of $G$ such that $E(S)=I(v)$. Observe that $S$ is a sink over $|I(v)|+1 $ vertices. In particular, $v$ is the unique stable inner vertex of~$S$.
We want to prove that $S$ is a dynamical region of $G$ and consequently $G=S$.
Since $G$ does not contain coherently oriented cycles, it is enough to prove that every vertex on the boundary of $S$ is unstable in $G$. 
Equivalently, if $w$ is a vertex of $S$ which is not $v$, then the edges in $E(G)\setminus E(S)$ incident to $w$ are all incoming.
If this was not the case, then $G$ would contain one of the graphs in Figure~\ref{fig:avoiding}, yielding a contradiction. This concludes the proof.
\endproof

We have some direct consequences:

\begin{corollary}\label{cor:modules}Let $G$ be a connected loopless MP-digraph that is not a coherently oriented cycle. Suppose that $G$  has a unique dinamical region. Then, the multipath matroid $M_G$ is isomorphic to the uniform matroid $U_{1, |E(G)|}.$
\end{corollary}
\begin{corollary}\label{cor:dinregions} Let $G$ be a connected loopless MP-digraph  that is not a coherently oriented cycle. Then, its dynamical modules are sinks, sources or linear graphs of length~1.
\end{corollary}
\begin{theorem}\label{thm:uniformprduct}
Let $G$ be an MP-digraph. Then, its multipath matroid $M_G$ is the direct sum of uniform matroids.
\end{theorem}
\proof Using Theorem~\ref{thm:components} and Remark~\ref{rmk:looped}, we can reduce to the case of $G$ connected without loops. If $G$ is a cycle, the assertion is Example~\ref{ex:cycle}. Otherwise, we can consider the decomposition of $G$ into dynamical regions. In view of Corollary~\ref{cor:modules} and of Corollary~\ref{cor:dinregions}, the multipath matroid of every dynamical module of $G$ is a uniform matroid. The statement now follows by Theorem~\ref{thm:modules}.
\endproof

A direct consequence of Theorem~\ref{thm:uniformprduct} is that the family of multipath matroids is self-dual -- as this holds true for the family of uniform matroids. Furthermore, we have proved that multipath matroids are graphic:

\begin{corollary} Every multipath matroid is regular and graphic. 
\end{corollary}
\proof Every uniform matroid is regular and graphic. Direct sum of regular and graphic matroid is again regular and graphic. The statement now follows from Theorem~\ref{thm:uniformprduct}. 
\endproof

Given a directed graph $G$ that satisfies the conditions of Theorem~\ref{thm:boom}, it is now easy to produce an undirected graph $G'$ such that $M_G$ is isomorphic to the graphic matroid of $G'$. 
In fact, up to considering coherently oriented cycles and loops as distinct connected components of $G$, we can reduce to the case of $G$ connected, loopless and without coherently oriented cycles. Furthermore, because of  Corollary~\ref{thm:modulesmatroid}, it is possible to reduce to the case of $G$ having a single dynamical region. We are in the hypotheses of  Proposition~\ref{prop:modules} and we can apply the classification of dynamical regions. More precisely, if $G$ is a linear graph, its multipath matroid coincides with its graphic one. Otherwise, if $G$ is a source or a sink with $n>1$ edges, then its multipath matroid is isomorphic to the graphic matroid of the (non directed) multigraph with set of vertices $V(G)=\{v, w\}$ and $n $ edges between $v$ and $w$. 
\section{Tutte Polynomials and Digraph colourings}
In this section we focus on the Tutte polynomial of multipath matroids. 
For a given MP-digraph~$G$, the main result of this section relates a specialization of the Tutte polynomial of~$M_G$ to certain colourings of $G$. Before introducing such digraph colourings, we investigate the properties of deletions and contractions of multipath matroids.

\subsection{Deletion-Contraction of Multipath Matroids} Let $G$ be a digraph and consider an edge $e=(v,w)$, $v \neq w$. We want to realize the deletion matroid $M_G \setminus e$ and the contraction $M_G/ e$ as multipath matroids of suitable digraphs. 

Recall that the digraph deletion 
$G \setminus e$ is the digraph obtained from $G$ by deleting the edge~$e$;  to be more precise, $G \setminus e$ is the digraph defined by the data $V(G \setminus e)=V(G)$ and $E(G \setminus e)=E(G) \setminus \{e\}$. We now need a slightly different notion of contraction of digraphs; which we call MP-contraction:

\begin{defn}[MP-Contraction Digraph]\label{defn:contractiongraphs} The digraph $G \mpslash e$ obtained contracting the edge $e = (v,w)$, with $v\neq w$, is the digraph defined as follows:
\begin{enumerate}[label={(CG\arabic*)}]
\item the set of vertices of $G \mpslash e $ is $V(G) \setminus \{v, w\}  $  with an additional vertex $x$;
\item if $(a,b) \in E(G)$ with $a, b  \in V(G) \setminus \{v, w\}  $ then $(a,b) \in E(G \mpslash e)$;
\item an edge of the form $(a, v)$ (resp.~$(w,b)$) of $G$ is replaced by an edge $(a, x)$ (resp.~$(x,b)$ in $G \mpslash e$;
\item Every edge of the form $(v, a)$, $a \neq w$,  or of the form $(b,w)$, $b \neq v$, in $G$ is replaced by a loop $L_{(v, a)}$ (resp.~$L_{(b,w)}$) incident to $x$ in~$G \mpslash e$.
\end{enumerate}
\end{defn}
 As an example, in Figure~\ref{fig:ALContracted} it is illustrated the graph obtained contracting the edge~$(v_1,v_0)$ of $A_3$, depicted in red in Figure~\ref{fig:AL}. Note also that the possible loops at $v$ or $w$, in Definition~\ref{defn:contractiongraphs}, are all sent to loops at $x$.
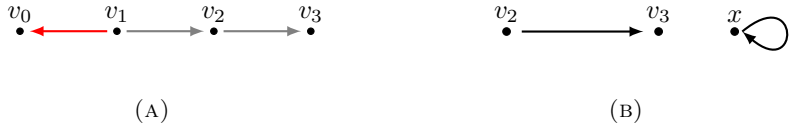
\begin{figure}[H]
	\centering
	\begin{subfigure}[b]{0.3\textwidth}
		\centering
	\begin{tikzpicture} [baseline=(current bounding box.center),scale =.85]
		\tikzstyle{point}=[circle,thick,draw=black,fill=black,inner sep=0pt,minimum width=2pt,minimum height=2pt]
		\tikzstyle{arc}=[shorten >= 8pt,shorten <= 8pt,->, thick]
		
		\node[above] (v0) at (0,0) {$v_0$};
		\draw[fill] (0,0)  circle (.05);
		\node[above] (v1) at (1.5,0) {$v_1$};
		\draw[fill] (1.5,0)  circle (.05);
		\node[above] (v2) at (3,0) {$v_{2}$};
		\draw[fill] (3,0)  circle (.05);
		\node[above] (v3) at (4.5,0) {$v_{3}$};
		\draw[fill] (4.5,0)  circle (.05);
		
		\draw[thick, red, -latex] (1.35,0) -- (0.15,0);
		\draw[thick, gray,  -latex] (1.65,0) -- (2.85,0);
		\draw[thick, gray,  -latex] (3.15,0) -- (4.35,0);
  \node at (0,-.5) {\phantom{A}};
	\end{tikzpicture}
		\caption{\phantom{A}}\label{fig:AL}
	\end{subfigure}\hspace{5em}
	\begin{subfigure}[b]{0.3\textwidth}
	\centering
	\begin{tikzpicture} [baseline=(current bounding box.center),scale =2]
		\tikzstyle{point}=[circle,thick,draw=black,fill=black,inner sep=0pt,minimum width=2pt,minimum height=2pt]
		\tikzstyle{arc}=[shorten >= 8pt,shorten <= 8pt,->, thick]
		
		 \node[above] (v0) at (3,0) {$x$};
		 \draw[fill] (3,0)  circle (.025);
		\node[above] (v2) at (1.5,0) {$v_2$};
		\draw[fill] (1.5,0)  circle (.025);
		\node[above] (v3) at (2.5,0) {$v_{3}$};
	\draw[fill] (2.5,0)  circle (.025);

		\draw[thick, -latex]  (3.05,0) to[in=320, out=40, loop] (3.05,0) ;
		\draw[thick, -latex]  (1.60,0.) -- (2.40,0) ;
	\end{tikzpicture}
				\caption{ \phantom{C }}\label{fig:ALContracted}
\end{subfigure}
	\caption{A linear graph (A) and its MP-contraction (B).}
	\label{fig:T}
\end{figure}

\begin{rmk}\label{rmk:delcontr}
Observe that if $G$ is a MP-digraph, then also $G \setminus e$ and $G \mpslash e$ are MP-digraphs, for each $e$ which is not a loop.
\end{rmk}

Our goal is to prove that the multipath matroid associated to the deletion/MP-contraction of a MP-digraph is the deletion/contraction of the corresponding multipath matroid. 
Let $e=(v,w)$ be an edge of $G$  that is not a loop. We denote by $\SubGraph_e (G) $ and $ \Mult_e (G)$ the sets of subgraphs and  multipaths, respectively, of $G$  that contain $e$. Consider the graph $\widetilde{\mathcal{C}}_e({G})$ obtained from $G\mpslash e$ by deleting all loops at the newly added vertex $x$. 
Observe that there is a map $$\widetilde{\mathcal{C}}_e\colon \SubGraph_e(G) \rightarrow \SubGraph\left(\widetilde{\mathcal{C}}_e({G})\right)$$
which sends each subgraph $H$ containing $e$ to $\widetilde{\mathcal{C}}_e({H})$.  
This map restricts to a map $$\mathcal{C}_e\colon \Mult_e (G) \rightarrow \Mult(\widetilde{\mathcal{C}}_e({G}))\ .$$
\begin{lemma}\label{lem:multcontraction} The map $\mathcal{C}_e$ induces a poset isomorphism between $ \Mult_e (G)$ and $\Mult(\widetilde{\mathcal{C}}_e({G}))$.
\end{lemma}

\proof
By construction of $\widetilde{\mathcal{C}}_e({G})$ it is clear that $\mathcal{C}_e$ is injective on $\Mult_e (G)$ and that it preserves inclusions of multipaths. We only need to prove that it is surjective. To this end, consider a multipath~$m' \in \Mult(\widetilde{\mathcal{C}}_e({G}))$, we shall show that $m'$ is the image of a multipath $m$.~\\%
\indent We start by assuming that no edge of~$m'$ is incident to $x$. Then, we can identify the edges of~$m'$ (if any) with the corresponding edges in $G$. In light of this identification, by setting~$m= m' \cup \{e\} $, it is immediate that $m\in \Mult_e (G)$ and that $m'= \mc{C}_e (m)$.~\\%
\indent We may therefore assume that there exists at least an edge $e' \in E(m')$ incident to $x$.
Since~$m'$ is a multipath, there is at most one other edge $e'' \in E(m')\setminus \{ e'\}$ incident to $x$ and the edges $e'$ and $e'' $ form a multipath in $\widetilde{\mathcal{C}}_e({G})$.
We shall assume that $e''$ exists, the same proof works also in the case there is no $e''$.
Removing the $e'$ and $e''$, from $m'$ yields a multipath $\overline{m}'$ in $\widetilde{\mathcal{C}}_e({G})$ with no edges incident to $x$.
Therefore, $\overline{m}'$ is the image $\mathcal{C}_e (\overline{m})$ of the multipath~$\overline{m} = \overline{m}' \cup \{e\}$ in $\Mult_e (G)$. 
Let $e_1\in E(G) \setminus \{ e\}$ be the edge corresponding (under the MP-contraction of $e$) to $e'$, and let $e_2$ be the edge in $G$ corresponding to $e''$. 
Denote by $m$ the subgraph of $G$ obtained by adding to $\overline{m}$ the edges $e_1$ and $e_2$.
Since $m' = \widetilde{\mc{C}}_e(m)$, to conclude it is enough to prove that~$m$ is a multipath in $G$.  
Observe that the edges $e_1$, $e$, and $e_2$ form a multipath in $G$. 
It is also not difficult to see that every vertex in $m$ has valence (in $m$) at most $2$. Thus, $m$ is a multipath if, and only if, it is not a coherently oriented cycle. But if~$m$ were a coherently oriented cycle, then also~$m\mpslash e$ would be a coherently oriented cycle. Which would be absurd since, by construction, we have that $m\mpslash e=\widetilde{\mathcal{C}}_e(m)=m'$ is a multipath in $\widetilde{\mathcal{C}}_e({G})$. This concludes the proof. 
\endproof

\begin{proposition}\label{prop:isodelcontr}
Let $G$ be a MP-digraph and let $M_G$ be the associated multipath matroid. Given an edge $e \in E(G)$ which is not a loop, we have: 
\[ M_{G \setminus e} = M_G \setminus e  \quad\text{and}\quad \quad M_{G \mpslashpedice e} = M_G / e\ .\]
\end{proposition}
\begin{proof}
The definitions of deletion of digraphs and matroids immediately give~$M_{G \setminus e} = M_G \setminus e$. 

We need to prove that $M_{G \mpslashpedice e} = M_G / e$. 
First, we have an identification, which stems directly from the definition of contraction matroid, between the poset of independents of $M_G / e$ with the poset $\Mult_{e}(G)$.
By Lemma~\ref{lem:multcontraction}, $\mathcal{C}_e$ gives a poset isomorphism between $\Mult_{e}(G)$ and $\Mult(\widetilde{\mathcal{C}}_e({G})) $.
The latter can be naturally identified, as poset, with  $\Mult(G\mpslash e)$.
Thus, composition of the above identifications and $\mathcal{C}_e$ gives a poset isomorphism $\psi$ between the independents of $M_G / e$ and~$\Mult(G\mpslash e)$.
A simple check shows that $\psi$ is in fact induced by the map $\Psi \colon E(G)\setminus \{e\} \rightarrow E(G \mpslash e)$ which associates to each edge the corresponding edge in $G\mpslash e$.  
\end{proof}

Comparing Proposition~\ref{prop:isodelcontr} with Theorem~\ref{thm:Tutteprop} immediately implies the following:

\begin{corollary}\label{cor:tuttegraph}
Let $G$ be a MP-digraph and $e$ an edge of $G$.
If $e$ is a loop in $M_G$, then 
\[T_{M_G}(x,y)= yT_{M_{G\setminus e}} (x,y) \ .\]
If $e$ is a coloop of $M_G$, we have 
\[T_{M_G}(x,y)= xT_{M_{G \mpslashpedice e}} (x,y)\ .\]
Finally, we have:
\[T_{M_G}(x,y)= T_{M_{G \setminus e}} (x,y) + T_{M_{G \mpslashpedice e}} (x,y) \ ,\]
if $e$ is neither a loop nor a coloop. 
\end{corollary}

The above corollary will be instrumental in linking the Tutte polynomial of multipath matroids with certain colourings of spanning subforests.

\subsection{Digraph colourings} 
In this subsection we introduce a special kind of colouring of directed graphs. Furthermore, we will relate the count of these colourings with an evaluation of the Tutte polynomial of multipath matroids.

\begin{defn}\label{defn:Dcolourings}
A \emph{flowing (vertex) $k$-colouring} for a digraph $G$ is a map $c \colon V(G) \to \{1, \dots , k\}$ such that: 
\begin{itemize}
\item if $(v,w)\in E(G)$ then $c(v) \neq c(w)$; 
\item if $(v, w)$ and $(v', w)$ are edges of $G$, then $c(v)=c(v')$;
\item if $( w,v)$ and $(w,v')$ are edges of $G$, then $c(v)=c(v')$. \end{itemize}
\end{defn}
We denote by $\tau_G(k)$ the number of $k$-colourings of the digraph $G$.

\begin{rmk}\label{rem:colorloop}
If $G$ has a loop, then  $\tau_G(k)=0$ for every $k$, by definition. If $G^{\mathrm{op}}$ is obtained from the digraph~$G$ by edge reversing, we have $\tau_{G}(k)=\tau_{G^{\mathrm{op}}}(k)$.
\end{rmk}
Observe that the flowing $k$-colourings of the coherently oriented cycle $C_n$ are the usual $k$-colourings of the underlying undirected graph $C_n^{\mathrm{ud}}$. The multipath matroid of $C_n$ is isomorphic to the graphic matroid of $C_n^{\mathrm{ud}}$, and consequenlty we have
\[k T_{M_{C_n}}(1-k, 0) =(-1)^{n-1} \chi_{C_n^{\mathrm{ud}}}(k) = (-1)^{n-1}\tau_{C_n}(k) \ , \]
where $\chi$ denotes the chromatic polynomial.
On the other hand, if  $A_n$ is the alternating graph, then we have $\tau_{A_n}(k) = k(k-1)$ whereas the number of $k$-colourings of the underlying unoriented graph is $\chi_{A_n^{\mathrm{ud}}}(k) = k(k-1)^{n-1}$.
While any undirected graph admits a $k$-colouring for $k$ large enough, this is not true for flowing $k$-colourings; as an illustrative example consider the digraph in Figure~\ref{fig:excolorazioni}, it is not difficult to check that it has no flowing $k$-colourings for any $k$. 
\begin{figure}[h]
    \centering
    \begin{tikzpicture}[thick,scale = 1.5]
        \node (a1) at (1,0) {};
        \node (a2) at (2,0) {};
        \node (b1) at (1,1) {};
        \node (b2) at (2,1) {};
        \node (c2) at (3,0.5) {};
        
        \node[below left]  at (1,0) {};
        \node[below right] at (2,0) {};
        \node[above] at (1,1) {};
        \node[above] at (2,1) {};
        
        \draw[fill] (a1) circle (0.03);
        \draw[fill] (b1) circle (0.03);
        \draw[fill] (a2) circle (0.03);
        \draw[fill] (b2) circle (0.03);
        \draw[fill] (c2) circle (0.03);

        \draw[-latex] (a1) -- (a2);
        \draw[-latex] (b1) -- (b2);
        \draw[-latex] (b2) -- (c2);
        \draw[-latex] (b1) -- (a1);
        \draw[-latex] (a2) -- (c2);
        
    \end{tikzpicture}
    \caption{A digraph with no flowing $k$-clolorings.}
    \label{fig:excolorazioni}
\end{figure}
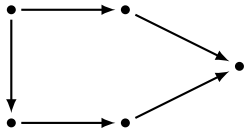

Let $G$ be a digraph, and $v\in V(G)$ a vertex.
Denote by $\od_G(v)$ and $\id_G(v)$ the outdegree and the indegree of $v$, respectively. 

\begin{lemma}\label{lem:treeisth}
Let $G$ be an MP-digraph without coherently oriented cycles. An edge $e=(v,w)$ of~$G$ is a coloop of $M_G$ if and only if  $\od_G(v)\leq 1$ and $\id_G(w)\leq 1$. 
In particular, if $e$ is a coloop, then the MP-contraction $G \mpslash e$ yields  the classical graph contraction $G/e$ of $e$ in~$G$.
\end{lemma} 
\proof 
Suppose  that $e$ is a coloop of $M_G$.  If $\od(v)$ is greater than~$1$, then there exists an edge $e'=(v, w')$ in $G$ with $w\neq w'$. Observe that any multipath $m$ containing $e'$ cannot contain $e$. 
This violates the condition of $e$ to be contained in every maximal independent set, leading to a contradiction. The case  $\id_G(w)> 1$ is completely analogous.

Suppose now that $\od_G(v)$ and $\id_G(w)$ are both smaller than or equal to $1$. This implies that~$G$ contains the digraph in Figure~\ref{fig:scopettone} as a subgraph, with $n, m \geq 0$. Proceeding as in the proof of \cite[Criterion B]{secondo}, the path poset of $G$ is then a cone over the path poset of $G \setminus e$. In particular 
$P(G) \cong P(G \setminus e) \times \{0 < 1\}$, where the elements of $P(G \setminus e) \times {1}$ can be identified with the multipaths of the form $m \cup e$, with $m \in P(G \setminus e)$. Observe that if $m \in P(G \setminus e) $ then $m \cup e > m $ in $P(G)$, and $r(m \cup e) > r(m) $. This implies that maximal rank elements of $P(G)$ are all contained in the subposet identified with $P(G \setminus e) \times \{1\}$. Since every element of this subposet contains $e$, then $e$ is a coloop.
\endproof

\begin{figure}[H]
    \centering
    \begin{tikzpicture}[thick]
        \node (a1) at (0,0.5) {};
        \node (a2) at (4,0.5) {};
        \node (b1) at (1.25,1) {};
        \node (b2) at (2.75,1) {};
        
        \node (c1) at (0,1.5) {};
        \node (c2) at (4,1.5) {};
        
        \node[below left]  at (0,0.5) {};
        \node[below right] at (4,0.5) {};
        \node[above] at (1.25,1) {};
        \node[above] at (2.75,1) {};
        \node[ ] at (0,1) {\raisebox{.5em}{$\vdots$}};
        \node[ ] at (4,1) {\raisebox{.5em}{$\vdots$}};
        
        \node[above left] at (0,0.75) {$m$};
        \node[above right] at (4,0.75) {$n$};
        
        \draw[fill] (a1) circle (0.03);
        \draw[fill] (b1) circle (0.03);
        \draw[fill] (c1) circle (0.03);
        \draw[fill] (a2) circle (0.03);
        \draw[fill] (b2) circle (0.03);
        \draw[fill] (c2) circle (0.03);

        \draw[-latex](b1) -- (b2);
        \draw[latex-] (a2) -- (b2);
        \draw[latex-] (b1) -- (a1);
        \draw[-latex] (c1) -- (b1);
        \draw[-latex] (b2) -- (c2);
        
    \end{tikzpicture}
    \caption{}
    \label{fig:scopettone}
\end{figure}

Lemma~\ref{lem:treeisth} will be used throughout the rest of this section, its first application is the following;

\begin{lemma}\label{lem:colouringcontraction}
Let $G$ be an MP-tree and suppose that there exists a coloop $e=(v,w)$ of $M_G$. Then, the equality
\begin{equation}\label{eq:colouringcontraction}\tau_G(k) = (k-1) \tau_{G / e}(k)\ ,\end{equation}
holds for each $k\neq 0$.
\end{lemma} 
\proof 
We proceed by induction on the number of edges, the case $|E(G)|=1$ being straightforward.
If $|E(G)|=n+1$, consider an edge $e'=(v',w')$, different from $e$, such that at least one between $v'$ and $w'$ is a leaf of $G$. Up to edge reversing, we may assume that $w'$ is such a leaf. 
Observe that, since $e$ is contained in every independent set of maximal cardinality, we have $v' \neq v $ by Lemma~\ref{lem:treeisth}. Consequently, if $w'$ is a leaf for $G$, then it is also a leaf for $G / e$.  
Moreover, in light of Lemma~\ref{lem:treeisth}, we have $\od_G(v')=\od_{G / e}(v')$ and $\id_G(w')=\id_{G / e}(w')$. 
Consider now the graph $\bar{G}$ such that $V(\bar{G})=V(G) \setminus \{w'\}$ and $E(\bar{G})=E(G) \setminus \{e'\}$. Observe that $\bar{G}$ is a tree and $e$ is a coloop of $M_{\bar{G}}$ -- since, by definition of $\bar{G}$, it is contained in every independent of maximal rank.
Then, by induction, we have  $\tau_{\bar{G}}(k) = (k-1) \tau_{\bar{G} / e}(k)$ for each $k\neq 0$.
Note that, if $\od_G(v')=1$, the only condition on the color of $w'$ in both $G$ and $G / e$ is that it must be different from the color of $v'$. Hence, we have
\[\tau_{G }(k) = (k-1)\tau_{\bar{G} }(k) =   
(k-1)^2 \tau_{\bar{G} /  e}(k)= (k-1)\tau_{G / e}(k)\ .  \]
Otherwise, if  $\od_G(v')>1$, by Definition~\ref{defn:Dcolourings} the color of $w'$ in both $G$ and $G / e$ is determined by the color of all other edges of the form $(v',w'')$. As a consequence, a flowing $k$-colouring of $\bar{G}$ and of $\bar{G} / e$ uniquely determines a flowing $k$-colouring of $G$ and of $G / e$, respectively. Therefore, we have
\[\tau_{G }(k) = \tau_{\bar{G}}(k) = (k-1)\tau_{\bar{G} / e}(k) = (k-1) \tau_{G / e}(k) \ , \]
concluding the proof.
\endproof

We are now ready to relate the Tutte polynomial with the numbers of flowing colourings.

\begin{theorem}\label{thm:colouringtrees} Let $G$ be an MP-tree and let $M_G$ be the associated multipath matroid. Then, 
\[k \, T_{M_G}(1-k,0)=(-1)^{r(G)}\tau_G(k) \ ,\]
where $r$ denotes the rank in $M_G$, and $r(G) = r(E(G))$.
\end{theorem}

\proof Observe that if $|E(G)|=0,1$, the assertion is true. We now proceed by induction on the number of edges~$|E(G)|$, and use the recursive relations of the Tutte polynomial. 

Suppose first that there exists an edge $e$ of $G$ that is a coloop of $M_G$.
 
Because of Lemma~\ref{lem:colouringcontraction} and using the inductive hypothesis, we have  
\begin{align*} \tau_G(k)&= (k-1)  \tau_{G/e}(k) \\ &= (k-1) (-1)^{r(G/e)} k T_{M_{G/e}}(1-k,0) \\
&=  - (-1)^{r(G/e)} k (1-k) T_{M_{G/e}}(1-k,0)\ .
\end{align*}
Then, the assertion follows in view of Corollary~\ref{cor:tuttegraph}.
Therefore, we can assume that there are no edges of $G$ that are coloops of $M_G$.

Up to reversing the orientation of $G$, we may assume that there is an edge $e=(v,w)$ with the vertex~$w$ a leaf.
By Lemma~\ref{lem:treeisth} there exists an edge $e' \in E(G)$ of the form $(v , w')$ for some $w' \in V(G)\setminus \{ w\}$. By construction, the MP-contraction $G\mpslash e$ has a loop, which we can identify with~$e'$, hence $T_{M_{G \mpslashpedice e}}(x,0)=0$. Therefore, using Corollary~\ref{cor:tuttegraph}, we get $T_{M_G}(x, 0)=T_{M_{G \setminus e}} (x ,0)$.  
Since the digraph~$G$ is a tree, $G \setminus e$ has two connected components $G_1$ and $G_2$, which are again MP-trees. By Theorem~\ref{thm:components} and Equation~\eqref{prop:DSTutte}, we have 
$T_{M_{G \setminus e}}(x, y ) = T_{M_{G_1}}(x,y) \, T_{M_{G_2}}(x,y)$, and using the  inductive hypothesis, we get
\begin{align*}\tau_{G \setminus e}(k) & =\tau_{G_1}(k) \, \tau_{G_2}(k) \\&= (-1)^{r(G_1)} k T_{M_{G_1}}(1-k,0) \, (-1)^{r(G_2)} k T_{M_{G_2}}(1-k,0)\\ &= (-1)^{r(G \setminus e )} k^2 T_{M_{G \setminus e}}(1-k, 0 ) \  .\end{align*} 
Proving the statement is now equivalent to prove that $k\tau_G(k)=\tau_{G \setminus e}(k)$ since, as we observed above, we have~$T_{G}(x,0) = T_{G\setminus e}(x,0)$.

Consider the graph $\bar{G}$ such that $V(\bar{G})=V(G) \setminus \{w\}$ and $E(\bar{G})=E(G) \setminus \{e\}$. By definition we have that  $G \setminus e = \bar{G} \sqcup \{w\}$ and then $k\tau_{\bar{G} }(k)= \tau_{G \setminus e}(k)$. Moreover, since by Definition~\ref{defn:Dcolourings} the color of $w$ in $G$ is determined by the color of $w'$, we have that
$ \tau_{G \setminus e}(k)= k \tau_{\bar{G} }(k) = k \tau_{G }(k) $,
and the theorem is proved. 
\endproof
Observe that, as a consequence of the proof of Theorem~\ref{thm:colouringtrees}, we have that if $G$ is an MP-tree and  $e=(v,w)$ is and edge of $G$ that is not coloop of $M_G$, then  $k \tau_G(k) =  \tau_{G \setminus e}(k)$.
Another straightforward consequence is the following;

\begin{corollary}\label{cor:colouringforests} Let $G$ be an MP-forest and let $M_G$ be the associated multipath matroid. Then, 
\[k^{p_0(G)}T_{M_G}(1-k,0)=(-1)^{r(G)}\tau_G(k)\]
where $p_0(G)$ is the number of connected components of $G$.
\end{corollary}
Our next goal is to extend what we showed so far to all digraphs satisfying the conditions of Theorem~\ref{thm:boom}.
For ease of notation, in the theorem below and its proof, for each graph $H$ we will write $r(H)$ \emph{en lieu} of $r_{M_H}(E(H))$.

\begin{theorem}\label{thm:colST} Let $G$ be a connected MP-digraph which is not a coherently oriented cycle. 
Then, for every spanning tree $S$ of $G$ of maximal rank (i.e.~$r(S)=r(G)$), the  relation 
\begin{equation}\label{eq:sptree}kT_{M_G}(1-k,0)=(-1)^{r(G)}\tau_S(k)\end{equation}
holds for all $k\geq 0$.
\end{theorem}
\proof We proceed by induction on the number of edges. If $|E(G)|=0,1$ the assertion is trivial.

Note that if $e$ is a coloop, then $r(G / e)=r(G)-1$. Furthermore, if $S$ is a subgraph of $G$, then $r(S) = r_{M_G}(E(S))$. 

First, suppose that there exists $e\in E(G)$ which is a coloop of $M_G$. Since $e$ is is contained in every multipath of maximal rank, $e$ is also contained in every spanning tree $S$ of $G$ of maximal rank.
Observe now that, by Lemma~\ref{lem:treeisth}, if $S$ is a spanning tree of maximal rank in $G$, then $S / e$ is a spanning tree of maximal rank in $G/e$ -- as $e$ is a coloop  of $G$. Moreover, since $r(S)=r(G)$ and $P(S)$ is a subposet of $P(G)$, then $e$ is a coloop also for $S$. By Lemma~\ref{lem:colouringcontraction}, using the inductive hypothesis, we obtain: 
\begin{align*}
\tau_S(k) &= (k-1) \tau_{S/e}(k) \\
&= (k-1) (-1)^{r(G / e)} k T_{M_{G/e}}(1-k,0) 
\\ &= (-1)^{r(G )} k (1-k) T_{M_{G/e}}(1-k,0) \ . \end{align*} 
The thesis now follows, in this case, by Corollary~\ref{cor:tuttegraph}.

Suppose that $G$ has no coloops. If $G$ is a tree, the statement is a consequence of Theorem~\ref{thm:colouringtrees}. Therefore, we can assume that $G$ is not a tree. Consider any spanning tree $S$ of $G$ of maximal rank, and an edge $e=(v,w) \in E(G) \setminus E(S)$. By the choice of $e$, $S$ is a spanning tree also for $G \setminus e$, which is, in particular, a connected graph. Furthermore, since $e$ is not a coloop, we have the equality $r(G)=r(G \setminus e)$ and, consequently, $S$ is a spanning tree of $G \setminus e$ of maximal rank. 
Moreover, in view of Lemma~\ref{lem:treeisth}, we have that at least one between $\od_G(v)$ and $\id_G(w)$ is greater than $1$. Then, $G\mpslash e$ contains a loop, and consequently we get $T_G(1-k,0)=T_{G \setminus e}(1-k,0)$. By inductive hypothesis we obtain that 
\[k \, T_G(1-k,0)= k T_{G \setminus e}(1-k,0) = (-1)^{r(G \setminus e) } \tau_S(k)=(-1)^{r(G) } \tau_S(k) \ , \]
which concludes the proof.
\endproof

If $G$ is disconnected, with connected components $G_1, \dots, G_k$, we say that a collection of directed trees $S_1, \dots, S_k$ is a \emph{spanning forest} for $G$ if $S_i$ is a spanning tree for $G_i$, for each $i$. 
In view of Theorem~\ref{thm:components} and of Proposition~\ref{prop:DSTutte}, the following corollary is a direct consequence of Theorem~\ref{thm:colST}.

\begin{corollary}\label{cor:tuttecolouring} Let $G$ be an MP-digraph without coherently oriented cycles and let $M_G$ be the associated multipath matroid. Then, for every spanning forest $S$ for $G$ of maximal rank, we have 
\begin{equation}\label{eq:spforest}k^{p_0(G)}T_{M_G}(1-k,0)=(-1)^{r(G)}\tau_S(k) \ ,\end{equation}
where $p_0(G)$ is the number of connected components of $G$.
\end{corollary}

The next example shows that equation~\eqref{eq:sptree} does not hold  if $S$ is a spanning tree of non-maximal rank.

\begin{example} The digraph $G$ in Figure~\ref{fig:excolorazioni} is an MP-digraph and has no flowing $k$-colourings for any~$k$. Since the maximal length of a multipath is $3$, we have  $r(G)=3$. It is possible to compute directly, for example using Theorem~\ref{thm:modules}, the Tutte polynomial of~$M_G$. In fact, $G$ decomposes into three dynamical modules: a linear sink, a linear source, and a single edge.  Therefore, the Tutte polynomial of~$M_G$ is $x(x+y)^2$. The directed spanning tree in Figure~\ref{fig:spanningtreeA} has maximal rank and has $k(k-1)^3$ flowing $k$-colourings. However, the spanning tree in Figure~\ref{fig:spanningtreeB} has rank~$2$ and has only $k(k-1)^2$ flowing $k$-colourings. 
\begin{figure}[h]
	\centering
	\begin{subfigure}[b]{0.3\textwidth}
	\centering
	 \begin{tikzpicture}[baseline=(current bounding box.center),scale =1.5]
	 	\tikzstyle{point}=[circle,thick,draw=black,fill=black,inner sep=0pt,minimum width=2pt,minimum height=2pt]
	 	\tikzstyle{arc}=[shorten >= 8pt,shorten <= 8pt,->, thick]
		
   \node (a1) at (1,0) {};
        \node (a2) at (2,0) {};
        \node (b1) at (1,1) {};
        \node (b2) at (2,1) {};
        \node (c2) at (3,0.5) {};
        
        \node[below left]  at (1,0) {};
        \node[below right] at (2,0) {};
        \node[above] at (1,1) {};
        \node[above] at (2,1) {};
        
        \draw[fill] (a1) circle (0.03);
        \draw[fill] (b1) circle (0.03);
        \draw[fill] (a2) circle (0.03);
        \draw[fill] (b2) circle (0.03);
        \draw[fill] (c2) circle (0.03);

        \draw[thick,-latex, red] (a1) -- (a2);
        \draw[thick,-latex, red] (b1) -- (b2);
        \draw[thick,-latex, red] (b2) -- (c2);
        \draw[thick,-latex, red] (b1) -- (a1);
        \draw[thick, gray, -latex] (a2) -- (c2);
        
	\end{tikzpicture}
				\caption{\phantom{A}}\label{fig:spanningtreeA}
\end{subfigure}
	\begin{subfigure}[b]{0.3\textwidth}
	\centering
	\begin{tikzpicture} [baseline=(current bounding box.center),scale =1.5]
		 \tikzstyle{point}=[circle,thick,draw=black,fill=black,inner sep=0pt,minimum width=2pt,minimum height=2pt]
		 \tikzstyle{arc}=[shorten >= 8pt,shorten <= 8pt,->, thick]
		
   \node (a1) at (1,0) {};
        \node (a2) at (2,0) {};
        \node (b1) at (1,1) {};
        \node (b2) at (2,1) {};
        \node (c2) at (3,0.5) {};
        
        \node[below left]  at (1,0) {};
        \node[below right] at (2,0) {};
        \node[above] at (1,1) {};
        \node[above] at (2,1) {};
        
        \draw[fill] (a1) circle (0.03);
        \draw[fill] (b1) circle (0.03);
        \draw[fill] (a2) circle (0.03);
        \draw[fill] (b2) circle (0.03);
        \draw[fill] (c2) circle (0.03);

        \draw[thick, gray, -latex] (a1) -- (a2);
        \draw[thick,-latex, red] (b1) -- (b2);
        \draw[thick,-latex, red] (b2) -- (c2);
        \draw[thick,-latex, red] (b1) -- (a1);
        \draw[thick,-latex, red] (a2) -- (c2);
        
	\end{tikzpicture}
				\caption{\phantom{B}}\label{fig:spanningtreeB}
\end{subfigure}
    \caption{}\label{fig:spanning tree}
\end{figure}
\end{example}
\section{The decategorification of multipath cohomology}

In this final section, we relate the Tutte polynomial of the multipath matroid with the (graded) Euler characteristic of multipath (co)homology.

Let $R$ be a commutative ring with identity, and $A$ an associative unital $R$-algebra.
Recall from Definition~\ref{def:multipaths} that $\Mult_i(G)$ denotes the set of multipaths of the digraph~$G$ with~$i$ edges. 
To each digraph $G$, we can associate the multipath cohomology of $G$. We briefly recall here its construction, referring to~\cite{primo} for the details. 

Given a multipath $ H\leq G$, to each connected component of~$H$ we associate a copy of~$A$. Then, we take the tensor product. More concretely, if $c_0, \dots , c_k$ are the connected components of~$H$, we define: 
\begin{equation}\label{eq:fun_obj}
	\mathcal{F}_{A}(   H)\coloneqq {A_{c_{0}}}  \otimes_R \cdots 
	\otimes_R  A_{c_k}\ ,\end{equation}
where $A_{c_{j}}$ is the copy of $A$ associated to the component $c_j$. Note that, while there might be a small issue in the choice of the ordering of the components, this can be fixed by choosing an ordering of the vertices of $G$. The isomorphism class of multipath homology does not depend on such order.
Then, we can define a (co)chain complex as 
\[ C^i_{\mu}(  G;A) \coloneqq \bigoplus_{\tiny\begin{matrix}
		   H\in P(  G)\\
		\ell(   H) = n
\end{matrix}}  {\mathcal{F}_A(H)}=\bigoplus_{H \in \Mult_i(G) }A^{\otimes p_0(H)}\ , 
\]
with differential $d_\mu$  defined using the multiplication in $A$. 

The homology of such cochain complex is called \emph{multipath cohomology}. A key feature of this construction is that, if $A$ is a graded $R$-algebra, then the cochain modules inherit a second grading, which is preserved by the differential.

Assume now that $A$ is a $\Z$-graded finitely-generated $R$-algebra, with $R$ an integral domain with trivial grading. The graded structure on~$A$ induces a natural structure of graded module on  each $C_\mu^i(G; A)$. We denote by $C_\mu^{i,j}(G, A)$ the ($R$-)submodule of $C_\mu^i(G, A)$ of degree $j$, and, as customary, we define the \emph{graded} (or \emph{quantum}) \emph{dimension} of $C_\mu^i(G, A)$ as follows: 
\[ \qdim (C_\mu^i(G; A)) \coloneqq \sum_{j} \dim_{Q(R)}(Q(R)\otimes_R C_\mu^{i,j}(G; A)) q^j \ ,\]
 where $Q(R)$ is the field of fractions of $R$.  The \emph{graded Euler-characteristic} is defined as
\[ \chi_\mu (G, q) = \sum_{i} (-1)^i \qdim (C_\mu^i(G, A)) \ .  \]
Recall the following general fact -- for instance, cf.~\cite[Theorem~2.44]{hatcher};

\begin{rmk}
Let $(C^{*},d)$ be a graded bounded cochain complex over $R$. Suppose that each $C^i$ is of finite rank and that the differential preserves the grading. Then, we have
\[ \sum_{i,j} (-1)^i  \dim_{Q(R)}\ (Q(R)\otimes_R C^{i,j}) q^j = \sum_{i,j} (-1)^i  \dim_{Q(R)}\ (Q(R)\otimes_R H^{i,j}(C^*,d)) q^j \ .\]
\end{rmk}

The following result is a consequence of a K\"unneth-type formula for the disjoint union of digraphs  -- see~\cite[Remark~3.1]{secondo} in the case of $A=\mathbb{K}$;

\begin{rmk}\label{rem:multChar}
Let $G_1$ and $G_2$ be graphs. Then, $\chi_\mu (G_1 \sqcup G_2, q)= \chi_\mu (G_1, q)  \chi_\mu (G_2, q)$.  \end{rmk}

 Before stating the main result of the section, observe that if we set $\alpha= \qdim(A)$, then the following formula for $ \chi_\mu (G, q)$ holds
\begin{equation}
\chi_\mu (G, \alpha) = \sum_{m \in \Mult(G) }(-1)^{|E(m)|} \alpha ^{p_0(m)} \ .
\end{equation}

Recall that for a digraph $G$, we denote by $r $ the rank function of its multipath matroid, and for a subdigraph $H$ of $G$ we denote by $n(H)=|E(H)|-r(H)$ the nullity of~$H$. 
    
\begin{theorem} Let $G$ be an MP-forest. Then,
\begin{equation*} \chi_\mu (G, \alpha) = (-1)^{r(G)} \alpha^{|E(G)| - r(G)} \alpha^{p_0(G)}T_{M_G}(1-\alpha, 1)  \ . \end{equation*}
\end{theorem}
\proof
Recall that the Tutte polynomial of the multipath matroid associated to the disjoint union of two MP-digraphs, say $G_1$ and $G_2$, is the product of the Tutte polynomials of $M_{G_1}$ and $M_{G_2}$, respectively -- cf.~Equation~\eqref{prop:DSTutte} and the proof of Theorem~\ref{thm:modulesmatroid}.
Furthermore, the quantities $r$, $E$, and $p_0$ are also additive under disjoint union of digraphs. Therefore, by Remark~\ref{rem:multChar}, we can reduce to the case of $G$ connected; i.e.~we assume $G$ to be a directed tree.
Observe that, as $G$ is a directed tree, then 
if $m$ is a multipath of $G$ we have $$p_0(m)=|E(G)|-|E(m)|+1 \ .$$
Consequently 
\begin{align*}
\chi_\mu (G, q) &= \sum_{m \in \Mult(G) }(-1)^{|E(m)|}\alpha^{p_0(m)} \\
&=\sum_{m \in \Mult(G) }(-1)^{|E(m)|}\alpha^{|E(G)|-|E(m)|+1} \\
 &= \alpha^{|E(G)|+1} \sum_{m \in \Mult(G)} \frac{(-1)^{|E(m)|}}{\alpha^{|E(m)|}}\ . 
\end{align*}
On the other hand, evaluating the Tutte polynomial
$T_{M_G}(x,y)$ at $x=1-q$, and~$y=1$, we have 
 \[
 T_{M_G}(1-q, 1)=  (-1)^{r(G)} q^{r(G)} \sum_{\tiny{\begin{matrix} A \subset E(G)\\ n(A)=0\end{matrix}}} \frac{(-1)^{r(A)}}{q^{r(A)}} \ . \]
Observe now that $n(A)=0$ if and only if $A$ is the set of edges of a multipath. Hence, we get
\begin{align*}
 T_{M_G}(1-q, 1)= (-1)^{r(G)} q^{r(G)}   \sum_{m \in \Mult(G)} \frac{(-1)^{r(m)}}{q^{r(m)}} \ .
\end{align*}
As $r(H)=|E(H)|$ if $H$ is a multipath of $G$, the statement follows after setting $q=\alpha$.
\endproof

\bibliographystyle{alpha}
\bibliography{bibliography}

\begin{thebibliography}{CCDTS24}

\bibitem[CCC22]{spri2}
L.~Caputi, D.~Celoria, and C.~Collari.
\newblock Monotone cohomologies and oriented matchings, 2022.
\newblock To appear on Homology, homotopy and Applications, ArXiv:2203.03476.

\bibitem[CCC23]{zbMATH07681103}
L.~Caputi, D.~Celoria, and C.~Collari.
\newblock Categorifying connected domination via graph {\"u}berhomology.
\newblock {\em J. Pure Appl. Algebra}, 227(9):22, 2023.
\newblock Id/No 107381.

\bibitem[CCDT]{primo}
L.~Caputi, C.~Collari, and S.~Di~Trani.
\newblock Multipath cohomology of directed graphs.
\newblock To appear on Algebraic \& Geometric Topology. Available on
  ArXiv:2108.02690.

\bibitem[CCDT23]{secondo}
L.~Caputi, C.~Collari, and S.~Di~Trani.
\newblock Combinatorial and topological aspects of path posets, and multipath
  cohomology.
\newblock {\em J. Algebr. Comb.}, 57(2):617--658, 2023.

\bibitem[CCDTS24]{terzo}
L.~Caputi, C.~Collari, S.~Di~Trani, and J.~Smith.
\newblock On the homotopy type of multipath complexes.
\newblock {\em Mathematika}, 70(1):e12235, 2024.

\bibitem[CCR24]{ramos}
L.~Caputi, C.~Collari, and E.~Ramos.
\newblock The weak categorical quiver minor theorem and its applications:
  matchings, multipaths, and magnitude cohomology, 2024.
\newblock ArXiv:2401.01248.

\bibitem[Hat00]{hatcher}
A.~Hatcher.
\newblock {\em {Algebraic topology}}.
\newblock Cambridge Univ. Press, Cambridge, 2000.

\bibitem[HGR05]{Helme-Guizon-Rong}
L.~Helme-Guizon and Y.~Rong.
\newblock A categorification for the chromatic polynomial.
\newblock {\em Algebr. Geom. Topol.}, 5:1365--1388, 2005.

\bibitem[HW17]{zbMATH06826922}
Richard Hepworth and Simon Willerton.
\newblock Categorifying the magnitude of a graph.
\newblock {\em Homology Homotopy Appl.}, 19(2):31--60, 2017.

\bibitem[JHR06]{zbMATH05118586}
Edna~F. Jasso-Hern{\'a}ndez and Yongwu Rong.
\newblock A categorification for the {Tutte} polynomial.
\newblock {\em Algebr. Geom. Topol.}, 6:2031--2049, 2006.

\bibitem[Oxl01]{OxGraphic}
J.G. Oxley.
\newblock {\em On the interplay between graphs and matroids}, page 199–240.
\newblock London Mathematical Society Lecture Note Series. Cambridge University
  Press, 2001.

\bibitem[Oxl11]{OxMat}
James Oxley.
\newblock {\em {Matroid Theory}}.
\newblock Oxford University Press, 02 2011.

\bibitem[Prz10]{Prz}
J.~H. Przytycki.
\newblock When the theories meet: {K}hovanov homology as {H}ochschild homology
  of links.
\newblock {\em Quantum Topol.}, 1(2):93--109, 2010.

\bibitem[Sto08]{zbMATH05313438}
Marko Sto{\v{s}}i{\'c}.
\newblock Categorification of the dichromatic polynomial for graphs.
\newblock {\em J. Knot Theory Ramifications}, 17(1):31--45, 2008.

\bibitem[SY18]{zbMATH06804238}
Radmila Sazdanovic and Martha Yip.
\newblock A categorification of the chromatic symmetric function.
\newblock {\em J. Comb. Theory, Ser. A}, 154:218--246, 2018.

\bibitem[SY24]{saito2024categorificationcharacteristicpolynomialmatroids}
T.~Saito and S.~Yamagata.
\newblock A categorification for the characteristic polynomial of matroids,
  2024.
\newblock ArXiv:2402.09851.

\bibitem[Tut54]{tutte}
W.~T. Tutte.
\newblock A contribution to the theory of chromatic polynomials.
\newblock {\em Canadian Journal of Mathematics}, 6:80–91, 1954.

\bibitem[TW12]{turnerwagner}
P.~Turner and E.~Wagner.
\newblock The homology of digraphs as a generalisation of {H}ochschild
  homology.
\newblock {\em Journal of Algebra and Its Applications}, 11(02):1250031, 2012.

\bibitem[V{\v{Z}}09]{VZ}
S.~T. Vre\'{c}ica and R.~T. {\v{Z}}ivaljevi\'{c}.
\newblock Cycle-free chessboard complexes and symmetric homology of algebras.
\newblock {\em European J. Combin.}, 30(2):542--554, 2009.

\bibitem[Wel10]{welshmath}
Dominic~JA Welsh.
\newblock {\em Matroid theory}.
\newblock Courier Corporation, 2010.

\bibitem[Whi86]{WhiteMaTh}
{\em Theory of Matroids}.
\newblock Encyclopedia of Mathematics and its Applications. Cambridge
  University Press, 1986.

\bibitem[Whi92]{WhiteMatApp}
{\em Matroid Applications}.
\newblock Encyclopedia of Mathematics and its Applications. Cambridge
  University Press, 1992.

\end{thebibliography}

\end{document}